\documentclass[11pt]{amsart} 
\usepackage{amsmath,amssymb, amsthm,amsbsy, amsfonts,amsthm,amscd, amssymb,amscd,mathrsfs} 
\usepackage[dvips]{graphicx}
\usepackage[usenames,dvipsnames]{pstricks} 
\usepackage{epsfig}
\usepackage[subnum]{cases}
\usepackage[colorlinks=true, linkcolor=magenta, citecolor=cyan, urlcolor=blue,hyperfootnotes=false]{hyperref}
\usepackage{pst-grad} 
\usepackage{pst-plot} 
\newcommand{\R}{{\mathbb R}}
\newcommand{\N}{{\mathbb N}} 
\newcommand{\T}{{\mathbb T}} 

\newcommand{\C}{{\mathbb C}}

\newcommand{\Lin}{\mathcal{L}}
\newcommand{\Hin}{\mathcal{H}}
\newcommand{\Win}{\mathcal{W}}

\newcommand{\m}{{\mathit m}}
\newcommand{\je}{{\mathit j}}
\newcommand{\Q}{{\mathcal Q}}

\renewcommand{\Re }{\mathrm {Re}}
\renewcommand{\Im }{\mathrm {Im}}

 \renewcommand{\geq }{\geqslant}
 \renewcommand{\leq }{\leqslant}

\newcommand{\round}[1]{{\left ( #1 \right )}}
\newcommand{\squad}[1]{{\left [ #1 \right ]}}
\newcommand{\norm}[2]{{\left\| #1 \right\|}_{#2}}

\usepackage{mathtools}
\DeclarePairedDelimiter{\abs}{\lvert}{\rvert}

\DeclarePairedDelimiter{\norma}{\lVert}{\rVert} 
\usepackage{braket}
\usepackage{color} 
\newcommand{\Rn}{{\mathbb R^n}}
\newenvironment{sistema}%
{\left\lbrace\begin{array}{@{}l@{}}}%
{\end{array}\right.}  
\newcommand{\supp}{{\rm supp}\,}

\numberwithin{equation}{section}
\newtheorem{theorem}{Theorem}[section]
\newtheorem{corollary}[theorem]{Corollary}
\newtheorem{lemma}[theorem]{Lemma}
\newtheorem{proposition}[theorem]{Proposition}
\theoremstyle{definition} 
\newtheorem{definition}[theorem]{Definition}

\newtheorem{remark}[theorem]{Remark}

\begin{document} 

\title[$H^1$-scattering for SCH-systems]{Decay and Scattering in energy space for the solution of weakly coupled Schr\"odinger-Choquard and  Hartree-Fock equations}


\author{M.~Tarulli}
\address{Mirko Tarulli: Dipartimento di Matematica, Universit$\grave{\text{a}}$ Degli Studi di Pisa
Largo Bruno Pontecorvo 5 I - 56127 Pisa. Italy.
Faculty of Applied Mathematics and Informatics, Technical University of Sofia, Kliment Ohridski Blvd. 8, 1000 Sofia, and IMI BAS, Acad. Georgi Bonchev Str., Block 8, 1113 Sofia, Bulgaria}
\email{mta@tu-sofia.bg}

\author{G.~Venkov}
\address{George Venkov: Faculty of Applied Mathematics and Informatics, Technical University of Sofia, Kliment Ohridski Blvd. 8, 1000 Sofia, Bulgaria}
\email{gvenkov@tu-sofia.bg}



\begin{abstract}
We prove decay with respect to some Lebesgue norms  for a class of Schr\"odinger equations with non-local nonlinearities by showing new Morawetz inequalities and estimates. As a byproduct, we obtain large-data scattering in the energy space for the solutions to the systems of $N$ defocusing Schr\"odinger-Choquard equations with mass-energy intercritical nonlinearities in any space dimension and of defocusing Hartree-Fock equations, for any dimension $d\geq3$. 
\end{abstract}
\maketitle

\section{Introduction}\label{sec:introduction}

The primary target of the paper is the study of the decaying and scattering properties of the solution to the following system of $N\geq 1$ nonlinear evolution equations in dimension $d\geq 1 $:
\begin{equation}\label{eq:nls}
\begin{cases}
i\partial_t \psi_j + \Delta \psi_j -
\sum_{k=1}^NG(\psi_j , \psi_k)=0,  \\ 
(\psi_j(0,\cdot))_{j=1}^N= (\psi_{j,0})_{j=1}^N \in H^1(\R^d)^N,
\end{cases}
\end{equation}
characterized by the nonlinearities 
\begin{align}\label{eq.nonlst}
G(\psi_j, \psi_k)=
\lambda_{jk} \squad{|x|^{-(d-\gamma_1)}*| \psi_k|^p} | \psi_j|^{p-2}  \psi_j\\
+
\beta\round {\squad{|x|^{-(d-\gamma_2)}*| \psi_k|^2  }\psi_j-\squad{|x|^{-(d-\gamma_2)}*\bar \psi_k \psi_j  } \psi_k}.
\nonumber
\end{align}
Here, for all $j,k=1,\dots,N$, $\psi_j=\psi_j(t,x):\R\times\R^d\to\C$, $(\psi_j)_{j=1}^N=(\psi_1,\dots, \psi_N)$ and  $\beta, \lambda_{jk}\geq0$
 are coupling parameters such that $\lambda_{jj}\neq 0$, $\beta=0$ if $p>2$.
 Henceforth, we name \eqref{eq:nls} as Schr\"odinger-Choquard (SCH) if $p\geq2$ and Hartree-Fock (HF) if $\lambda_{jk}=0$ for all $j,k=1,\dots, N, \, N\geq 2$ in \eqref{eq.nonlst}.
 We will require that the nonlinearity parameters $p, \gamma_1$ satisfy the following relations
\begin{align}
\label{eq:base}
0<\gamma_1<d,& \ \
 2\leq p<p^*(d), \ \  p^*(d)=
\begin{cases}
\infty \,  & \text{if} \ \  d=1,2, \\
\frac {d+\gamma_1}{d-2}  \,    &\text{if} \ \ d\geq3,
\end{cases}
\\
&p> p_{*}(d), \ \  p_{*}(d)=\frac{d+\gamma_1+2}d,
\label{eq:baseII}
\end{align}
that is the  $L^2$-supercritical and $H^1$-subcritical regime. We shall assume also that $\max{\round{0, d-4}}<\gamma_2<d$ if $\beta\neq 0$. The system \eqref{eq:nls} enjoys two important conserved quantities: we have the mass
\begin{align}\label{eq.mass}
  M(\psi_j)(t)=\int_{\R^d}|\psi_j(t)|^2\,dx,
\end{align}
 for any $j=1,\dots,N$ and the energy,
 \begin{equation}\label{eq:energy}
\begin{split}
  E(\psi_{1},\dots,\psi_{N})
   =  \sum_{j=1}^N\int_{\R^d}  \abs{\nabla{\psi_j}}^2
 + \frac{1}{2p}\sum_{j,k=1 }^N\lambda_{jk} \int({|x|^{-(n-\gamma_1)}}\ast
|\psi_k|^{p})|\psi_j|^{p} dx
\\
+\frac {\beta}2\sum_{j,k=1 }^N \int_{\R^d}\round {\squad{|x|^{-(d-\gamma_2)}*| \psi_k|^2  }|\psi_j|^2-\squad{|x|^{-(d-\gamma_2)}*\bar \psi_k \psi_j  } \psi_k \bar \psi_j}dx.
\end{split}
   \end{equation}

The equation \eqref{eq:nls} has a strong physical meaning and its role is important in many models of mathematical physics.
In fact, the special case of the Hartree-Newton equation, that is when $d=3$, $p = 2$, $N=1$, $\gamma_1 = 2$ and  $\beta=0$ in \eqref{eq.nonlst}, was variously introduced in the scenario of quantum mechanics in order to represent the mean-field limit of large systems of bosons (the so-called Bose-Einstein condensates) by considering the self-interactions of the such charged particles. We suggest, in this direction, \cite{ES}, \cite{Len}, \cite{LR} and the references therein. 
About the HF equation, that is the case when $d=3$, $N\geq 2$ and $\lambda_{jk}=0,\, j,k=1\dots, N$ in \eqref{eq.nonlst}, it was applied in \cite{Fo} for certain approximations in the theory of one component, for portraying an exchange term resulting from Pauli's principle as well as for describing the fermions as an approximation of the equation overlooking the impact of their fermionic nature. Other relevant papers about this topic are \cite{BJPSS} and \cite{BSS2} (see also the references inside). Furthermore, in \cite{FrLe} the Hartree-Fock equation was fundamental for developing models of white dwarfs. Turning to the SCH equation, the case of  $d=3$, $\beta=0$, $p, \gamma_1$ as in \eqref{eq:base} and $N=1$ in \eqref{eq:nls} was introduced to sketch an electron trapped in its own hole, as showed in \cite{ChS} and \cite{CSV} and very recently in \cite{Pen}, to describe self-gravitating matter together with quantum entanglement and quantum information effects. Morivated by this and by \cite{CLM}, where the general case of systems of interacting particles is studied, we carry on with the analysis of the decay properties of the solution to  \eqref{eq:nls} unfolding large-data scattering in $H^1(\R^d)^N$ for the Schr\"odinger-Choquard and the Hartree-Fock systems on $N$ particles. By pursuing the ideas initially introduced in \cite{CT} for systems of Nonlinear Schr\"odinger equations with local nonlinearities (see also \cite{TzVi}, \cite{Vis} for the single NLS and \cite{Ta} for the fourth-order NLS), we introduce relevant breakthroughs extending the theory to the non-local setting. Namely, the system \eqref{eq:nls} is translation invariant so we can set up either the Morawetz viriel and action, or their bilinear analogues. As a main outcome, we are able to present new Morawetz identities, interaction Morawetz identities and their associated inequalities for \eqref{eq:nls}. The succeeding step is to localize the Morawetz inequalities on space-time slabs having $\R^d$-cubes as space components, utilizing again the translation invariance of the equation and of all the estimates involved. We say, that at level of localized frame, the dichotomy between local and non-local interactions breaks down: the convolution functions appearing in the interaction Morawetz can be handled in the same manner as if we are treating pure power nonlinearities. The corresponding localized estimates accomplish a contradiction argument which implies the decay of $L^r$-norms of the solutions $(\psi_j(t,x))_{j=1}^N$, provided that $2<r\leq2d/(d-2)$, for $d\geq 3$ and, if $p>2$, $2<r<\infty$ for $d=1,2$, with $r=\infty$ included for $d=1$. Let us underline that our approach guarantees the possibility to deal with the SCH in low
spatial dimension $d=1,2$, bypassing the techniques of \cite{Nakanishi1}.
Now, this peculiar behaviour, jointly with a suitable reformulation of the theory developed in \cite{Ca}, bears to the asymptotic completeness and existence of the wave operators in the energy space $H^1(\R^d)^N$ for solution to \eqref{eq:nls}. We point out now the novelties introduced in our paper. Looking at the Schr\"odinger-Hartree equation (that is, SCH with $p=2$) and at the HF systems in dimension $d\geq3$, one knows that the aforementioned decay and the consequent scattering are similarly achieved in several papers like \cite{GiOz}, \cite{GiVel3}, \cite{NaOz}, \cite{Wa}, where the pseudo-conformal technique was successfully applied once one assumes the initial data laying in a weighted energy space. We improve all these results by selecting the initial data
 in $H^1(\R^d)^N$ only, showing a similar decay of the solution to \eqref{eq:nls} in the range $\max{\round{0, d-4}}< \gamma_1,\gamma_2< d$. We refine also the decay property of the solutions and simplify some of the results released in \cite{GiVel2} and \cite{GiVel}, where the scattering in the energy space for the  Schr\"odinger-Hartree equation is acquired, for $\max{\round{0, d-4}}< \gamma_1< d-2$, without imposing further regularity to the initial data. Let us move to the case of the defocusing SCH given by \eqref{eq:nls} with $d\geq 1$, $N\geq 1$, $p>2$ in \eqref{eq.nonlst}. We earn in this setting the full decay of the solution of the system \eqref{eq:nls}, the existence of the scattering states and that the wave operators are well-defined and bijective in the energy-space $H^1(\R^d)$. Moreover, all such properties are transposed to the special case of $N=1$, that is 
\begin{equation}\label{eq.Hart}
  \begin{sistema}
i \partial_t \psi + \Delta \psi =\lambda ({|x|^{-(n-\gamma_1)}}\ast
|\psi|^{p})|\psi|^{p - 2} \psi,\\
\psi(0,x)=\psi_0(x),
  \end{sistema}
\end{equation}
with $\lambda>0$. Currently, we are unaware of alike results, so we emphasize that ours are new in the whole literature. This explains the reason why we can not supply any kind of references.
\\ 

The first main target of this paper is the following.

\begin{theorem}\label{decay}
 Let $(\psi_{j})_{j=1}^N \in\mathcal C(\R, H^1(\R^d)^N)$ be the unique global solution to \eqref{eq:nls} with $p>2,$ and $\gamma_1$ in \eqref{eq.nonlst} such that \eqref{eq:base} holds. Then, for all $j=1,\dots,N$, one has the decay property
\begin{align}\label{eq:decay1}
\lim_{t\rightarrow \pm \infty} \|\psi_j(t)\|_{L^r(\R^d)}=0,
\end{align}
with $2<r\leq 2d/(d-2),$ for $d\geq 3$, with $2<r<+\infty,$ for $d=2$ and with $2<r\leq+\infty,$ for $d=1$.
Let $d\geq 3$, if $(\psi_{j})_{j=1}^N \in\mathcal C(\R, H^1(\R^d)^N)$ is the unique global solution to \eqref{eq:nls} with $p=2$ and $\max{\round{0, d-4}}<\gamma_1,\gamma_2<d$ in \eqref{eq.nonlst}, then \eqref{eq:decay1}  remains valid along with $2<r\leq 2d/(d-2).$
\end{theorem}

The second main result concerns the scattering of the solution in the energy space. 

\begin{theorem}\label{thm:mainNLCH}
  Assume $d\geq 1$ and $p>2$, $\gamma_1$  such that \eqref{eq:base}, \eqref{eq:baseII} hold or $d\geq3$ and $p=2$, $\max{\round{0, d-4}}<\gamma_1,\gamma_2<d-2,$ in \eqref{eq.nonlst}. Let $(\psi_{j})_{j=1}^N \in\mathcal C(\R, H^1(\R^d)^N)$ be the unique global solution to \eqref{eq:nls}, then:
  \begin{itemize}
  \item\emph{(asymptotic completeness)}  
    There exists $(\psi_{j,0}^{\pm})_{j=1}^N \in H^1(\R^d)^N$ such that for all $j=1,\dots,N$ 
    \begin{equation}\label{eq:scattering}
      \lim_{t\to\pm\infty}\left\|\psi_j(t,\cdot)-e^{it\Delta}\psi_{j,0}^\pm(\cdot)\right\|_{H^1(\R^d)}=0.
    \end{equation} 
  \item\emph{(existence of wave operators)} For every $(\psi_{j,0}^{\pm})_{j=1}^N \in H^1(\R^d)^N$ there exists unique 
    initial data $(\psi_{j,0})_{j=1}^N \in H^1(\R^d)^N,$ such that the
    global solution to \eqref{eq:nls} 
    $(\psi_{j})_{j=1}^N \in\mathcal C(\R, H^1(\R^d)^N)$  satisfies
    \eqref{eq:scattering}.
  \end{itemize}
\end{theorem}

\begin{remark}
We observe that \eqref{eq:base} and \eqref{eq:baseII} force to some restrictions on $\gamma_1$ above. As long as $p>2$, we need only that $p^*(d)>2$, for $d\geq 3$. This is equivalent to the condition $\gamma_1>\max\round{0, d-4}$. In dimensions $d=1,2$ the fact that $p_{*}(d)>2$ and $p^*(d)=\infty$, grants the full range $0<\gamma_1<d$. Unlike above, if $p=2$  then one has to require $p_{*}(d)<2<p^*(d)$. This compels to the conditions $\max{\round{0, d-4}}<\gamma_1,\gamma_2<d-2,$ which are mandatory for the well-posedness and the asymptotic completness.

\end{remark}

Then, Theorem \ref{thm:mainNLCH} leads directly to other consequences. First we have the immediate one for the
SCH equation:

\begin{corollary}\label{thm:mainNLCH2a}
  Let $d\geq 1$ and $p, \gamma_1$ as in \eqref{eq:base}. Then, if $\psi_0 \in H^1(\R^d)$, the unique global solution  $\psi\in\mathcal C(\R, H^1(\R^d))$ to \eqref{eq.Hart} is such that:
  \begin{itemize}
  \item if $p>2$, the decay property
    \begin{align}\label{eq:decayC}
\lim_{t\rightarrow \pm \infty} \|\psi(t, \cdot)\|_{L^r(\R^d)}=0,
\end{align}
is verified for $2<r\leq2d/(d-2),$ $d\geq 3$, for $2<r<+\infty,$ $d=2$ and for  $2<r\leq+\infty,$ $d=1$;
    \item if $p>\max\round{2, p_{*}(d)}$,
  the scattering occurs, i.e. there exists $\psi_0^{\pm} \in H^1(\R^d)$ such that  
    \begin{equation}\label{eq:scatteringSC}
      \lim_{t\to\pm\infty}\left\|\psi(t,\cdot)-e^{-it\Delta}\psi_{0}^\pm(\cdot)\right\|_{H^1(\R^d)}=0.
    \end{equation}
    \end{itemize}
\end{corollary}

In the Schr\"odinger-Hartree and HF systems framework we have:

\begin{corollary}\label{thm:mainNLCH2b}
  Let $d\geq 3$, $p=2$ and $\max{\round{0, d-4}}<\gamma_1,\gamma_2<d-2$ in \eqref{eq.nonlst}.
  Then, if $(\psi_{j,0}^{\pm})_{j=1}^N \in H^1(\R^d)^N$, the unique global solution $(\psi_{j})_{j=1}^N \in\mathcal C(\R, H^1(\R^d)^N)$ to \eqref{eq:nls} is such that:
  \begin{itemize}
  \item the decay property
    \begin{align}\label{eq:decayHF}
\lim_{t\rightarrow \pm \infty} \|\psi_j(t, \cdot)\|_{L^r(\R^d)}=0,
\end{align}
is fulfilled for $2<r\leq2d/(d-2)$;
    \item the scattering occurs, i.e. there exists $(\psi_{j,0}^{\pm})_{j=1}^N \in H^1(\R^d)^N$ such that  
    \begin{equation}\label{eq:scattering0}
      \lim_{t\to\pm\infty}\left\|\psi_j(t,\cdot)-e^{it\Delta}\psi_{j,0}^\pm(\cdot)\right\|_{H^1(\R^d)}=0.
    \end{equation}
    \end{itemize}
\end{corollary}

\begin{remark}
The foregoing corollary summarizes different results. In the case $p=2$ and $\beta=0$, we get  
\eqref{eq:decayHF} displaced for the system of $N\geq 1$ coupled Schr\"odinger-Hartree equations and if  $N\geq 2$ and $\lambda_{jk}=0$ for all $ j,k=1\dots, N$, we have the same decay property for the solution of the HF equation. 
Once \eqref{eq:decayHF} is proved, we can construct the scattering operators in the energy space.

\end{remark}

The literature related to these subjects is not so wide and according to our knowledge, Morawetz and interaction Morawetz estimates were available for systems of NLS for the first time in \cite{CT} and successively in \cite{Ta}. We come to an end by itemizing briefly some other achievements, different from the already cited ones, which regard particular versions of \eqref{eq:nls}.
The well-posedness for the Schr\"odinger-Hartree equation, both local and global, was examined in \cite{Ca}, \cite{Lieb} and \cite{MXZ2} while the existence of the standing waves was discussed in \cite{Li}. The scattering in the focusing critical case was examined in \cite{MXZ} and the blow up of the solutions in the focusing framework, in \cite{MXZ3} (we suggest the references contained therein also). As we said, little is known for the single SCH. On the other hand we cite here  \cite{FeYU} for the well-posedness for the single SCH and \cite{Li}
for the well-posedness and blow-up in the case of SCH perturbed by an inverse square potential. We remind  \cite{GV}, in which local and global well-posedness, existence of standing waves and blow up solutions were investigated for \eqref{eq.Hart} with $\gamma_1=2$ for $p=(d+4)/d$ in the focusing case $\lambda<0$. We mention also \cite{BAGS}, \cite{GS} and \cite{MV}, for more general informations about the solitary waves solution of the focusing \eqref{eq.Hart}. In closing, we recall that scattering for the focusing SCH in $d\geq 3$, was earned in \cite{Ar} and in \cite{ArRo} for large radial data and small data, respectively.

\subsection*{Outline of paper.} 
After some preliminaries in Section \ref{Prelim}, through the Section \ref{InterMor} we build, in Lemma \ref{le:morV} and Lemma \ref{le:intmorNL}, the Morawetz inequalities and their bilinear counterpart, respectively. The principal target of the Section \ref{MainThm1} is to unveil the decay of some Lebesgue norms of the solutions to the systems \eqref{eq:nls}, which is a fundamental property for catching the scattering states and is included in Proposition \ref{decay}. Finally, all the remaining scattering theory associated to \eqref{eq:nls} takes place in Section \ref{NLSscat}. The last section is the Appendix \ref{appen}, in which a localized Gagliardo-Nirenberg inequality, an ancillary tool used extensively beside the paper, is obtained.

\section{Preliminaries}\label{Prelim}
We indicate by $L_x^r$ the Lebesgue space $L^r(\R^d)$, and by $W^{1,r}_x$ and $H^1_x$ the  inhomogeneous Sobolev spaces $W^{1,r}(\Rn)$ and $H^1(\Rn)$, respectively (for more details see \cite{Ad}).
For any $N\in \N $, we also define $ \Lin_x^r= L^r(\R^d)^N$ and introduce 
the Sobolev spaces $\Win^{1,r}_x=W^{1,r}(\R^d)^N$ and $\Hin^1_x=H^1(\R^d)^N.$
From now on and in the sequel we adopt the following notations: for any two positive real numbers $a, b,$ we write $a\lesssim b$  (resp. $a\gtrsim b$) to denote $a\leq C b$ (resp. $Ca\geq b$), with $C>0,$ we unfold the constant only when it is essential. 
We recall also some of the results concerning the well-posedness for \eqref{eq:nls} already available, such as \cite{ArRo} \cite{FeYU}, \cite{Li} for the SCH and as
\cite{GiVel}, \cite{NaOz}, \cite{Wa} in the HF framework. Then we can summarize them as: 

\begin{proposition}\label{ConsLaw}
Let $d\geq 1$ and assume  \eqref{eq.nonlst} is such that $p>2$, $\gamma_1$  satisfy \eqref{eq:base} or $p=2$, $\max{\round{0, d-4}}<\gamma_1,\gamma_2<d$. Then for all $(\psi_{j,0})_{j=1}^N \in \Hin^1_x$ there exists a unique 
global solution $(\psi_j)_{j=1}^N \in \mathcal C(\R,\Hin^1_x) $ to \eqref{eq:nls}, moreover 
\begin{align}
  &M(\psi_j)(t)=\norma{\psi_j(0)}_{L^2_x} 
   \label{eq:macons}
  \end{align}
for all $j=1,\dots,N$ and 
  \begin{align}
  &E(\psi_1(t),\dots,\psi_N(t))=E(\psi_1(0),\dots,\psi_N(0))
  \label{eq:enecons},
\end{align}
with $E(\psi_1(t),\dots,\psi_N(t))$ as in \eqref{eq:energy}.
\end{proposition}
The proposition above can be obtained by standard energy method (see Theorem 3.3.9 and Remark 3.3.12 in \cite{Ca}) combined with the inequality
 \begin{equation*}
\int_{\R^d}({|x|^{-(d-\gamma)}}\ast |\psi|^{p})|\psi|^{p} dx\lesssim  \|\psi\|^{2p}_{L_x^{\frac{2pd}{d+\gamma}}}
\lesssim \|\psi\|^{2p}_{H_x^1},
\end{equation*}
for $p\in [(d+\gamma)/d,(d+\gamma)/(d-2)]$ ($p\in [(d+\gamma)/d,\infty)$, if $d=1,2$), as well as the defocusing nature of the system.

\section{Morawetz identities and nonlinear interaction Morawetz inequalities}\label{InterMor}
We provide, thorough this section, the fundamental tools for the proof of our first main theorem. 
We start by obtaining Morawetz-type identities, which are comparable to the ones holding for the single NLS. From now on we hide the variable $t$ for simplicity, spreading it out only when necessary.
Moreover, we find suitable to set up the following notations: given a function $h\in H^1(\R^d,\C)$, we denote by
\begin{equation}\label{eq:notation}
 \mathit m_h(x):=|h(x)|^2,
  \qquad
  \mathit {j}_h(x):=\Im\left[\overline{h}\nabla h(x)\right],
\end{equation}
the mass and momentum densities, respectively. We have the Morawetz identities for non-local nonlinearities.
\begin{lemma}\label{le:morV}
Let $d\geq1$ and $(\psi_j)_{j=1}^N \in\mathcal C(\R,H^1(\R^d)^N)$ be as in Proposition \ref{ConsLaw},
 let $a=a(x):\R^d\to\R$ be a sufficiently regular and decaying function, and indicate by
\begin{equation*}
   \mathcal V(t):=\sum_{j=1}^N\int_{\R^d}a(x) \, \m_{\psi_j}(x)\,dx.
\end{equation*}
The following identities hold:
\begin{align}
\dot {\mathcal V}(t)
=\sum_{j=1}^N \int_{\R^d}a(x) \dot \m_{\psi_j}(x)\,dx
=2\sum_{j=1}^N \int_{\R^d} \je_{\psi_j}(x)  \cdot\nabla a(x)\,dx
\label{eq:morV}
\\
\ddot {\mathcal V}(t)
=\sum_{j=1}^N \int_{\R^d}a(x)\ddot \m_{\psi_j}(x)\,dx
\label{eq:morV2}
\\
=\sum_{j=1}^N\left[-\int_{\R^d} \m_{\psi_j(x)}(x)\Delta^2a(x)\,dx
+4  \int_{\R^d}\nabla \psi_j(x)D^2a(x)\cdot\nabla \psi_j(x)\,dx \right]
\nonumber
\\
+\frac{2(p-2)}{p}
\sum_{j,k=1}^N \lambda_{jk}\int_{\R^d} \Delta a(x)\squad{|x|^{-(d-\gamma_1)}*| \psi_k|^p} | \psi_j(x)|^{p}\,dx ,
\nonumber
\\
-
\frac 4 p\sum_{j,k=1}^N \lambda_{jk}\int_{\R^d}\nabla a(x)\cdot \nabla \squad{|x|^{-(d-\gamma_1)}*| \psi_k|^p} | \psi_j(x)|^{p}\,dx ,
\nonumber
\\
-
2\beta\sum_{j,k=1}^N \int_{\R^d}\nabla a(x)\cdot \nabla F(x,\psi_j, \overline \psi_j, \psi_{k}, \overline \psi_{k} )\,dx ,
\nonumber
\end{align}
with $\beta=0$ if $p>2$,
\begin{align}\label{eq:nonMHF}
F(x,\psi_j, \overline \psi_j, \psi_{k}, \overline \psi_{k})=\\
\squad{|x|^{-(d-\gamma_2)}*| \psi_k|^2 }|\psi_j(x)|^2-\squad{|x|^{-(d-\gamma_2)}*\bar \psi_k \psi_j  } \psi_k(x) \overline \psi_j(x),
\nonumber
\end{align}
 for any $j,k=1,\dots,N$, $D^2a\in\mathcal M_{d\times d}(\R^d)$ is the Hessian matrix of $a$ and $\Delta^2a=\Delta(\Delta a)$ the bi-laplacian operator.
\end{lemma}
\begin{proof}
We will proceed similarly to \cite{CT} (see also \cite{Ta}, \cite{TzVi}).  We shall assume that
 $(\psi_j)_{j=1}^N$ is a smooth solution to \eqref{eq:nls}, taking into account that the case 
  $(\psi_j)_{j=1}^N \in\mathcal C(\R,H^1(\R^d)^N)$ can be established by a density argument (we cite, for example, \cite{GiVel}). The equation \eqref{eq:morV} is simple to derive. We carry out some details for providing \eqref{eq:morV2} only. By means of an integration by parts and thanks to \eqref{eq:nls}, we have
  \begin{align}\label{eq:MorId1}
  2\sum_{j=1}^N\partial_t \int_{\R^d} \je_{\psi_j}(x)\cdot \nabla a(x) \,dx  \\
     =- 2\sum_{j=1}^N\Im  \int_{\R^d} \partial_t \psi_j(x) [\Delta a(x)\bar \psi_j(x)+2\nabla a(x)\cdot \nabla \bar \psi_j(x)]\,dx 
    \nonumber\\
     =2\sum_{j=1}^N\Re \int_{\R^d} i \partial_t \psi_j(x) [\Delta a(x)\bar \psi_j(x)+2\nabla a(x)\cdot \nabla \bar \psi_j(x)]\,dx
     \nonumber\\ 
    =2\sum_{j=1}^N\Re \int_{\R^d} \Big[-\Delta \psi_j(x) +G(\psi_j , \psi_k)\Big][\Delta a(x)\bar \psi_j(x)+2\nabla a(x)\cdot \nabla \bar \psi_j(x)]\,dx.
    \nonumber
      \end{align}
First, one can get 
  \begin{align}\label{eq:MorId2}
    2\sum_{j=1}^N\Re\int_{\R^d} -\Delta \psi_j(x) [\Delta a(x)\bar \psi_j(x)+2\nabla a(x)\cdot \nabla \bar \psi_j(x)]\,dx\\
    = - \sum_{j=1}^N\int_{\R^d} \Delta^2 a(x) \abs{\psi_j(x)}^2\,dx + 4\sum_{j=1}^N\int_{\R^d} \nabla \psi_j(x) D^2 \psi_j(x)\nabla \bar \psi_j(x)\,dx.
    \nonumber
  \end{align}
Furthermore we obtain
\begin{align}\label{eq.id1}
    2  \sum_{j,k=1}^N \Re \int_{\R^d}G(\psi_j , \psi_k)
    [\Delta a(x)\bar \psi_j(x)+2\nabla a(x)\cdot \nabla \bar \psi_j(x)]\,dx \\
    = 2\sum_{j,k=1}^N \Re  \int_{\R^d}G_1(\psi_j , \psi_k)[\Delta a(x)\bar \psi_j(x)+2\nabla a(x)\cdot \nabla \bar \psi_j(x)]\,dx
    \nonumber\\
    +2\sum_{j,k=1}^N \Re  \int_{\R^d}G_2(\psi_j , \psi_k) [\Delta a(x)\bar \psi_j(x)+2\nabla a(x)\cdot \nabla \bar \psi_j(x)]\,dx,
    \nonumber
\end{align}
with
\begin{align}\label{eq.id2}
    2   \sum_{j,k=1}^N\Re \int_{\R^d} G_1(\psi_j , \psi_k)
    [\Delta a(x)\bar \psi_j(x)+2\nabla a(x)\cdot \nabla \bar \psi_j(x)]\,dx \\
 =
2\sum_{j,k=1}^N \lambda_{jk}\int_{\R^d} \Delta a(x)\squad{|x|^{-(d-\gamma_1)}*| \psi_k|^p} | \psi_j(x)|^{p}\,dx
\nonumber\\
+4\sum_{j,k=1}^N \lambda_{jk}\Re\int_{\R^d} \nabla a(x)\squad{|x|^{-(d-\gamma_1)}*| \psi_k|^p} \cdot| \psi_j(x)|^{p-2}\psi_j(x)\nabla \bar\psi_j(x)\,dx
\nonumber\\
=2\round{1-\frac 2p}\sum_{j,k=1}^N \lambda_{jk}\int_{\R^d} \Delta a(x)\squad{|x|^{-(d-\gamma_1)}*| \psi_k|^p} | \psi_j(x)|^{p}dx
\nonumber\\
-\frac 4p\sum_{j,k=1}^N \lambda_{jk}\int_{\R^d}\nabla a(x)\cdot \nabla \squad{|x|^{-(d-\gamma_1)}*| \psi_k|^p} | \psi_j(x)|^{p}\,dx
\nonumber
\end{align}
and
\begin{align}\label{eq.id3}
    2  \sum_{j,k=1}^N \Re \int_{\R^d} G_2(\psi_j , \psi_k)
    [\Delta a(x)\bar \psi_j(x)+2\nabla a(x)\cdot \nabla \bar \psi_j(x)]\,dx \\
 =
2\beta\sum_{j,k=1}^N \int_{\R^d} \Delta a(x)\squad{|x|^{-(d-\gamma_2)}*| \psi_k|^2 }|\psi_j(x)|^2\,dx
\nonumber\\
+2\beta\sum_{j,k=1}^N \Re\int_{\R^d} 2\nabla a(x)\squad{|x|^{-(d-\gamma_2)}*| \psi_k|^2 } \cdot\psi_j(x) \nabla \bar \psi_j(x)\,dx
\nonumber\\
-2\beta\sum_{j,k=1}^N \int_{\R^d} \Delta a(x)\squad{|x|^{-(d-\gamma_2)}*\bar \psi_k \psi_j  } \psi_k(x) \bar \psi_j(x)\,dx
\nonumber\\
-2\beta\sum_{j,k=1}^N\Re\int_{\R^d} 2\nabla a(x)\squad{|x|^{-(d-\gamma_2)}*\bar \psi_k \psi_j  } \psi_k(x) \cdot \nabla \bar\psi_j(x)\,dx.
\nonumber
\end{align}
An integration by parts of the the second term on the r.h.s. of the above identity \eqref{eq.id3} enhances to
\begin{align}\label{eq.id4}
2 \beta\sum_{j,k=1}^N \Re\int_{\R^d} \nabla a(x)\squad{|x|^{-(d-\gamma_2)}*| \psi_k|^2 }\cdot \nabla | \psi_j(x)|^2\,dx
\\
=-2\beta\sum_{j,k=1}^N \int_{\R^d} \Delta a(x)\squad{|x|^{-(d-\gamma_2)}*| \psi_k|^2 } | \psi_j(x)|^2\,dx
\nonumber\\
-2 \beta\sum_{j,k=1}^N \int_{\R^d} \nabla a(x)\cdot \nabla\squad{|x|^{-(d-\gamma_2)}*| \psi_k|^2 } | \psi_j(x)|^2\,dx.
\nonumber
\end{align}
By a further integration by parts, one has for the last term in \eqref{eq.id3}, instead,
\begin{align}\label{eq.id5}
4\beta\sum_{j,k=1}^N\Re\int_{\R^d}\int_{\R^d} \nabla a(x)\frac{\bar\psi_k(y) \psi_k(x)}{|x-y|^{(d-\gamma_2)} }\cdot\nabla \psi_j(y) \bar\psi_j(x)\,dxdy
\\
=-2\beta\sum_{j,k=1}^N \int_{\R^d} \Delta a(x)\squad{|x|^{-(d-\gamma_2)}*\bar \psi_k \psi_j } \psi_k(x) \bar \psi_j(x)\,dx
\nonumber\\
-2\beta \sum_{j,k=1}^N \int_{\R^d} \nabla a(x)\cdot \nabla\squad{|x|^{-(d-\gamma_2)}*\bar \psi_k \psi_j } \psi_k(x) \bar \psi_j(x)\,dx.
\nonumber
\end{align}
We can utilize now \eqref{eq.id2} in combination with \eqref{eq.id3}, \eqref{eq.id4} and \eqref{eq.id5} to rewrite 
 \eqref{eq.id1} as
\begin{align}\label{eq.MorId3}
    2  \sum_{j=1}^N \Re \int_{\R^d}G(\psi_j , \psi_k)
    [\Delta a(x)\bar \psi_j(x)+2\nabla a(x)\cdot \nabla \bar \psi_j(x)]\,dx \\
    = \frac{2(p-2)}{p}
\sum_{j,k=1}^N \lambda_{jk}\int_{\R^d} \Delta a(x)\squad{|x|^{-(d-\gamma_1)}*| \psi_k|^p} | \psi_j(x)|^{p}\,dx ,
\nonumber
\\
-
\frac 4p\sum_{j,k=1}^N \lambda_{jk}\int_{\R^d}\nabla a(x)\cdot \nabla \squad{|x|^{-(d-\gamma_1)}*| \psi_k|^p} | \psi_j(x)|^{p}\,dx ,
\nonumber
\\
-
2\beta\sum_{j,k=1}^N \int_{\R^d}\nabla a(x)\cdot \nabla F(x,\psi_j, \overline \psi_j, \psi_{k}, \overline \psi_{k} )\,dx ,
\nonumber
\end{align}
with $F(x,\psi_j, \overline \psi_j, \psi_{k}, \overline \psi_{k} )$ as in \eqref{eq:nonMHF}. Then the above identities  \eqref{eq:MorId2} and \eqref{eq.MorId3} bring us to the proof of \eqref{eq:morV}.
\end{proof}

One can now apply the previous lemma for proving the following interaction Morawetz identities and inequalities for non-local nonlinearities.

\begin{lemma}\label{le:intmorNL}
Let $(\psi_j)_{j=1}^N \in\mathcal C(\R, H^1(\R^d)^N)$ be as in Proposition \ref{ConsLaw}, $a=a(|x|):\R^d\to\R$ be a convex radial, sufficiently regular and decaying function. Indicate by $a^{\star}=a^{\star}(x,y):=a(|x-y|):\R^{2d}\to\R$ and by
\begin{equation*}
  \mathcal I(t):=\sum_{j,\ell=1}^N\int_{\R^d}\int_{\R^d}a^{\star}(x,y) \m_{\psi_j}(x)\m_{\psi_\ell}(y)\,dxdy.
\end{equation*}
The following holds:
\begin{align}
  \dot {\mathcal I}(t)
  =2\sum_{j,\ell=1}^N\int_{\R^d}\int_{\R^d}\je_{\psi_j}(x)\cdot\nabla_xa^{\star}(x,y) \, \m_{\psi_\ell}(y)\,dxdy,
  \label{eq:intmorV}
  \\
 N^{C}_{(p,a^{\star})}(t)+N^{HF}_{(2,a^{\star})}(t)+R^{C}_{(p,a^{\star})}(t)\leq \ddot {\mathcal I}(t),
  \label{eq:intmorV2}
  \end{align}
where
\begin{align}  \label{eq:nonlinC}
N^{C}_{(p,a^{\star})}(t)=
\\
\sum_{\substack{j,k,\ell=1}}^N\widetilde\lambda_{jk}\int_{\R^d}\int_{\R^d}
 \Delta_xa^{\star}(x,y) \squad{|x|^{-(d-\gamma_1)}*| \psi_k(x)|^p}|\psi_j(x)|^{p}\m_{\psi_\ell}(y)\,dxdy,
 \nonumber
\end{align}
with $ \widetilde\lambda_{jk}=4\lambda_{jk}(p-2)/p,$
\begin{align}  \label{eq:nonlinRC}
R^{C}_{(p,a^{\star})}(t)=
\\
-\frac 8p \sum_{\substack{j,k,\ell=1}}^N\lambda_{jk}\int_{\R^d}\int_{\R^d}
 \nabla_xa^{\star}(x,y) \cdot \nabla_x \squad{|x|^{-(d-\gamma_1)}*| \psi_k|^p}|\psi_j(x)|^{p}\m_{\psi_\ell}(y)\,dxdy,
 \nonumber
\end{align}
\begin{align}  \label{eq:nonlinHF}
N^{HF}_{(2,a^{\star})}(t)=
\\
-4\beta\sum_{\substack{j,k,\ell=1}}^N\int_{\R^d}\int_{\R^d}
 \nabla_xa^{\star}(x,y)\cdot \nabla_x F(x,\psi_j, \overline \psi_j, \psi_{k}, \overline \psi_{k} ) \m_{\psi_\ell}(y)\,dxdy,
 \nonumber
\end{align}
with $\beta=0$ if $p>2$ and $F(x,\psi_j, \overline \psi_j, \psi_{k}, \overline \psi_{k} )$ as in \eqref{eq:nonMHF}.
\end{lemma}
\begin{proof}
 As formerly done, we prove the identities for a smooth solution of \eqref{eq:nls}, 
 moving to the general case $(\psi_j)_{j=1}^N \in\mathcal C(\R,H^1(\R^d)^N)$ by an usual density argument.
  First, we point out that \eqref{eq:intmorV},  because of the symmetry of the function $a^{\star}(x,y)=a(|x-y|)$, is equivalent to 
  \begin{align*}
    \dot {\mathcal I}(t) =2 \sum_{j,\ell=1}^N \int_{\R^d}\int_{\R^d}a^{\star}(x,y)
  \dot \m_{\psi_j}(x)\m_{\psi_\ell}(y)\,dxdy.
  \end{align*}
  Hence, \eqref{eq:intmorV} is straightforward  from \eqref{eq:morV} and Fubini's Theorem.
  We differentiate w.r.t. time variable working out now the equality
  \begin{align}\label{eq:mo1}
  \mathcal{\ddot I}(t)
  \\
  =
 -4\sum_{j,\ell=1}^N
 \Re \int_{\R^d}\int_{\R^d}
 \m_{\psi_\ell}(y) i\partial_t (\overline \psi_j(x)\nabla_x\psi_j(x))\cdot\nabla_xa^{\star}(x,y)\,dxdy
   \nonumber \\
    -2\sum_{j,\ell=1}^N \Re \int_{\R^d}\int_{\R^d}
 i\partial_t m_{\psi_j}(x)   \overline \psi_\ell(y)\nabla_y\psi_\ell(y)\cdot\nabla_ya^{\star}(x,y)\,dxdy
 \nonumber\\
  -2\sum_{j,\ell=1}^N \Re \int_{\R^d}\int_{\R^d} i\partial_t m_{\psi_\ell}(y)  \overline \psi_j(x)\nabla_x\psi_j(x)\cdot\nabla_xa^{\star}(x,y)\,dxdy
 \nonumber\\
 :=\mathcal {II}_1(t)+\mathcal {II}_2(t).
    \nonumber
  \end{align}
 By the identity \eqref{eq:morV2}, Fubini's Theorem and the
  symmetry of $a^{\star}(x,y)$ we achieve
  \begin{align}
  \mathcal {II}_1(t) =
  -2\sum_{j,\ell=1}^N\int_{\R^d}\int_{\R^d}
 \Delta^2_xa^{\star}(x,y) m_{\psi_j}(x)m_{\psi_\ell}(y)\,dxdy&
   \label{eq:2Main}
  \\
+ \sum_{\substack{j,k,\ell=1}}^N\widetilde\lambda_{jk}\int_{\R^d}\int_{\R^d}
 \Delta_xa^{\star}(x,y) \squad{|x|^{-(d-\gamma_1)}*| \psi_k(x)|^p}|\psi_j(x)|^{p}\m_{\psi_\ell}(y)\,dxdy,
 \nonumber\\
  -\frac 8p\sum_{\substack{j,k,\ell=1}}^N\lambda_{jk}\int_{\R^d}\int_{\R^d}
 \nabla_xa^{\star}(x,y) \cdot \nabla_x \squad{|x|^{-(d-\gamma_1)}*| \psi_k|^p}|\psi_j(x)|^{p}\m_{\psi_\ell}(y)\,dxdy,
 \nonumber\\
 -4\beta\sum_{\substack{j,k,\ell=1}}^N\int_{\R^d}\int_{\R^d}
 \nabla_xa^{\star}(x,y) \nabla_xF(x,\psi_j, \overline \psi_j, \psi_{k}, \overline \psi_{k} ) \m_{\psi_\ell}(y)\,dxdy,
 \nonumber
\end{align}
The first term of \eqref{eq:2Main}  above arises from the linear part of the equation, while the
other terms are connected to the nonlinearity contained in the equation. 
The linear term can be modified as follows
  \begin{equation}\label{eq:equivPV}
    \begin{split}
    &-2\sum_{j,\ell=1}^N\int_{\R^d}\int_{\R^d}\Delta_x^2a^{\star}(x,y) m_{\psi_j}(t,x)m_{\psi_\ell}(t,y)\,dxdy\\
    =&2\sum_{i,j,\ell=1}^N\int_{\R^d}\int_{\R^d} \partial_{x_i}\partial_{y_i}\Delta_x a^{\star}(x,y) m_{\psi_j}(t,x)m_{\psi_\ell}(t,y)\,dxdy\\
    =&2\sum_{j,\ell=1}^N\int_{\R^d}\int_{\R^d}\Delta_x a^{\star}(x,y) \nabla_{x} m_{\psi_j}(t,x)\cdot\nabla_{y}m_{\psi_\ell}(t,y)\,dxdy,
    \end{split}
  \end{equation}
by an integration by parts and taking again advantage of the property $\partial_{x_i}a^{\star}=-\partial_{y_i}a^{\star}.$
At the end, we have
\begin{align}\label{eq:mo2}
 \mathcal{II}_1(t) =  N^{C}_{(p,a^{\star})}(t)+N^{HF}_{(2,a^{\star})}(t)+R^{C}_{(p,a^{\star})}(t)\\
 +
 2\sum_{j,\ell=1}^N\int_{\R^d}\int_{\R^d}\Delta_x a^{\star}(x,y) \nabla_{x} m_{\psi_j}(t,x)\cdot\nabla_{y}m_{\psi_\ell}(t,y)\,dxdy.
 \nonumber
\end{align}
As well, by the Morawetz identities \eqref{eq:morV}, \eqref{eq:morV2} and Fubini's Theorem we get
  \begin{align}\label{eq:equivPV2}
  \mathcal{II}_2(t)  =
  &
4\sum_{j,\ell=1}^N\int_{\R^d}\int_{\R^d}
  \nabla \psi_j(x)D^2_xa^{\star}(x,y)\nabla\overline \psi_j(x)
  m_{\psi_\ell}(y)\,dxdy
   \nonumber\\
  &
   +4\sum_{j,\ell=1}^N\int_{\R^d}\int_{\R^d}
  m_{\psi_j}(x)\nabla\psi_\ell(y)D^2_ya^{\star}(x,y)\nabla\overline \psi_\ell(y)
  \,dxdy
  \\
  &
  +8\sum_{j,\ell=1}^N\int_{\R^d}\int_{\R^d}
  j_{\psi_j}(x)D^2_{xy}a^{\star}(x,y)\cdot j_{\psi_\ell}(y)\,dxdy,
  \nonumber
  \end{align}
 here we applied, at this point, the symmetry of $D^2 a^{\star}$ to drop the real part condition in the first two 
summands on the r.h.s. of the identity above.  Once more, the fact that $\partial_{x_i}a^{\star} = -\partial_{y_i}a^{\star}$
allows us to reshape  \eqref{eq:equivPV2} as 
   \begin{align}\label{eq:equivPV3}
  \mathcal{II}_2(t)= 4\sum_{j,\ell=1}\int_{\R^d}\int_{\R^d} \nabla_{y}\psi_\ell(y)D^2_{x}a(|x-y|)\nabla_{y}\overline \psi_\ell(y)|\psi_j(x)|^2\,dxdy&
 \nonumber \\
  4\sum_{j,\ell=1}\int_{\R^d}\int_{\R^d} \nabla_{x}\psi_j(x)D^2_{x}\phi(|x-y|)\nabla_{x}\overline \psi_j(x)|\psi_\ell(y)|^2\,dxdy&
  \\
  -8\sum_{j,\ell=1}\int_{\R^d}\int_{\R^d}
  \Im(\overline \psi_j(x)\nabla_{x}\psi_j(x))D^2_{x}a(|x-y|)
  \Im(\overline \psi_\ell(y)\nabla_{y}\psi_\ell(y))\,dxdy&
  \nonumber
  \end{align}
  and lastly to
  \begin{align}\label{eq:equivPV4}
    \mathcal{II}_2(t)
    \\
    =2\sum_{j,\ell=1}\int_{\R^d}\int_{\R^d}
 \left(H^1_{j\ell} D^2_{x}a(|x-y|)\overline{H^1_{j\ell}}+H^2_{j\ell} D^2_{x}a(|x-y|)\overline{H^2_{j\ell}}\right)\,dxdy,&  
 \nonumber
 \end{align}
where we set
   \begin{align*}
    H^1_{j\ell} 
    &
    :=\psi_j(t,x)\nabla_{y}\overline{\psi_\ell(t,y)}+\nabla_{x}\psi_j(t,x)\overline{\psi_\ell(t,y)},
    \\
    H^2_{j\ell}
    &
    :=\psi_j(t,x)\nabla_{y}\psi_\ell(t,y)-\nabla_{x}\psi_j(t,x)\psi_\ell(t,y).
  \end{align*}
  Thus by \eqref{eq:equivPV4}, and since $a$ is a convex function one achieves $\mathcal{II}_{2}(t)\geq 0$, for any $t\in\R$.
  We claim further that
  \begin{align}\label{eq:equivPVN}
 2\sum_{j,\ell=1}^N\int_{\R^d}\int_{\R^d}\Delta_x a^{\star}(x,y) \nabla_{x} m_{\psi_j}(t,x)\cdot\nabla_{y}m_{\psi_\ell}(t,y)\,dxdy+\mathcal{II}_{2}(t)\geq 0.
\end{align}
In fact by (see for instance to \cite{PlVe}), we obtain
 \begin{equation*}
    \begin{split}
    &-2\sum_{j,\ell=1}^N\int_{\R^d}\int_{\R^d}\Delta_x^2a^{\star}(x,y) m_{\psi_j}(x)m_{\psi_\ell}(y)\,dxdy\\
    =&-2\sum_{j,\ell=1}^N\int_{\R^d}\int_{\R^d}\Delta_x\Delta_y a(|x-y|) m_{\psi_j}(x)m_{\psi_\ell}(y)\,dxdy\\
    =&2\sum_{j,\ell=1}^N\int_{\R^d}\int_{\R^d} \nabla_{x} m_{\psi_j}(x)D_x^2a(|x-y|) \cdot\nabla_{y}m_{\psi_\ell}(y)\,dxdy,
    \end{split}
  \end{equation*}
then the l.h.s. of \eqref{eq:equivPVN} becomes equal to
\begin{align}\label{eq:equivPVN2}
 2\sum_{j,\ell=1}^N\int_{\R^d}\int_{\R^d} \nabla_{x} m_{\psi_j}(x)D_x^2a(|x-y|) \cdot\nabla_{y}m_{\psi_\ell}(y)\,dxdy
 \\
 +4\sum_{j,\ell=1}\int_{\R^d}\int_{\R^d}
 \left(H^1_{j\ell} D^2_{x}a(|x-y|)\overline{H^1_{j\ell}}+H^2_{j\ell} D^2_{x}a(|x-y|)\overline{H^2_{j\ell}}\right)\,dxdy.
 \nonumber
\end{align}  
A straight computation displays
\begin{align}\label{eq:equivPVN3}
 2\nabla_x|\psi_j(x)|^2D_x^2a(|x-y|) \cdot \nabla_y   |\psi_\ell(y)|^2
 \\
 +2H^1_{j\ell} D^2_{x}a(|x-y|)\overline{H^1_{j\ell}}+2H^2_{j\ell} D^2_{x}a(|x-y|)\overline{H^2_{j\ell}}
\nonumber \\
=2H^1_{j\ell} D^2_{x}a(|x-y|)\overline{H^1_{j\ell}}-2H^2_{j\ell} D^2_{x}a(|x-y|)\overline{H^2_{j\ell}}
 \nonumber\\
 +2H^1_{j\ell} D^2_{x}a(|x-y|)\overline{H^1_{j\ell}}+2H^2_{j\ell} D^2_{x}a(|x-y|)\overline{H^2_{j\ell}}=4H^1_{j\ell} D^2_{x}a(|x-y|)\overline{H^1_{j\ell}}\geq 0.
 \nonumber
\end{align}
where we employed along the calculation that the matrix $D^2_xa^{\star}=-D^2_{xy}a^{\star}=D^2_xa^{\star}$ is symmetric.
Gathering together \eqref{eq:equivPVN2} and \eqref{eq:equivPVN3} we have that \eqref{eq:equivPVN} is satisfied. With this last inequality in mind, we sum now 
$\mathcal{II}_{1}(t)$ with $\mathcal{II}_{2}(t)$ realizing that
\begin{align}\label{eq:equivPVNFin}
\mathcal{\ddot I}(t)= 2\sum_{j,\ell=1}^N\int_{\R^d}\int_{\R^d}\Delta_x a^{\star}(x,y) \nabla_{x} m_{\psi_j}(t,x)\cdot\nabla_{y}m_{\psi_\ell}(t,y)\,dxdy
\\
+\mathcal{II}_{2}(t)+ N^{C}_{(p,a^{\star})}(t)+N^{HF}_{(p,a^{\star})}(t)+R^{C}_{(p,a^{\star})}(t)
\nonumber\\
\geq  N^{C}_{(p,a^{\star})}(t)+N^{HF}_{(2,a^{\star})}(t)+R^{C}_{(p,a^{\star})}(t),
\nonumber
\end{align}
that is the desired \eqref{eq:intmorV2}.
  \end{proof}

The proof of Lemma \ref{le:intmorNL} accomplishes also the following proposition.

\begin{proposition}\label{prop:intmorbil}
 Let $(\psi_j)_{j=1}^N \in\mathcal C(\R, H^1(\R^d)^N)$  be as in Proposition \ref{ConsLaw} and
 $N^{C}_{(p,a^\star)}(t),\, N^{HF}_{(2,a^\star)}(t)$, $R^{C}_{(p,a^\star)}(t)$ as in Lemma \ref{le:intmorNL}, then the following holds.
 \begin{itemize}
  \item (Low regularity Morawetz interaction inequality)
   \begin{align}
  \ddot {\mathcal I}(t)
  \label{eq:intmor2aa}
  \geq  N^{C}_{(p,a^{\star})}(t)+N^{HF}_{(2,a^{\star})}(t)+R^{C}_{(p,a^{\star})}(t)
  \\
+2\sum_{j,\ell=1}^N\int_{\R^d}\int_{\R^d}
  \Delta_xa^{\star}(x,y) \nabla_x\m_{\psi_j}(x)\cdot\nabla_y \m_{\psi_\ell}(y)\,dxdy.
  \nonumber
\end{align}
 \item (High regularity Morawetz interaction inequality)
\begin{align}\label{eq:intmor2bb}
 \ddot {\mathcal I}(t) \geq N^{C}_{(p,a^{\star})}(t)+N^{HF}_{(2,a^{\star})}(t)+R^{C}_{(p,a^{\star})}(t)
 \\
   -2\sum_{j,\ell=1}^N\int_{\R^d}\int_{\R^d}
 \Delta^2_xa^{\star}(x,y) \m_{\psi_j}(x) \m_{\psi_\ell}(y)\,dxdy.
 \nonumber
\end{align}
\end{itemize}
\end{proposition}
\begin{proof}
We will supply the proof in few lines. Here we shall make use of \eqref{eq:equivPVNFin} along with \eqref{eq:equivPV}, arriving to the inequalities \eqref{eq:intmor2aa} and \eqref{eq:intmor2bb}.
\end{proof}

A direct consequence of Lemma \ref{le:intmorNL} is that we can prove the following:

\begin{proposition}\label{dimg3}
  Assume $d\geq 1$, $p>2$
  and let $(\psi_j)^N_{j=1}\in\mathcal C(\R,H^1(\R^d)^N)$ be as in Proposition \ref{ConsLaw}. Then, selecting $a^{\star}(x,y)=|x-y|$, one has the global estimate
\begin{align}
\int_{\R}N^{C}_{(p,a^{\star})}(t)\,dt 
 \leq C\sum_{j=1}^N\|\psi_{j,0}\|^4_{H^1_x},
  \label{eq:stimaMain}
\end{align}
with $N^{C}_{(p,a^{\star})}(t)$ as in \eqref{eq:nonlinC}.
Moreover, let be $\Q_{\tilde x}^d(r)=\tilde x+[-r,r]^d$, with $r>0$ and $\tilde x\in\R^d$, one gets the following localized
estimates: for $d\geq2$,
\begin{align} \label{eq:stima0a}
 \sum_{j,k,\ell=1}^N\widetilde\lambda_{jk}\int_{\R}\sup_{\tilde x\in\R^d}\int_{(\Q_{\tilde x}^d(r))^3}|\psi_j(t,x)|^{p}| |\psi_\ell(t,y)|^{2}|\psi_k(t,z)|^{p}
\,dxdydzdt,
  \nonumber\\
 \leq C\sum_{j=1}^N\|\psi_{j,0}\|^4_{H^1_x},
  \end{align}
  where $ \widetilde\lambda_{jk}=4\lambda_{jk}(p-2)/p$ and $(\Q_{\tilde x}^d(r))^3=\Q_{\tilde x}^d(r)\times \Q_{\tilde x}^d(r)\times \Q_{\tilde x}^d(r)$;\\
for $d=1$, 
  \begin{align} \label{eq:stima0b}
\sum_{j, k, \ell=1}^N\widetilde\lambda_{jk}\int_{\R}\sup_{\tilde x\in\R}\int_{(\Q_{\tilde x}^1(r))^2}|\psi_j(t,x)|^{p}|\psi_\ell(t,x)|^{2}|\psi_k(t,z)|^{p}
 \,dx \,dz\,dt 
  \nonumber\\
 \leq C\sum_{j=1}^N\|\psi_{j,0}\|^4_{H^1_x},
  \end{align}
   with $(\Q_{\tilde x}^1(r))^2=\Q_{\tilde x}^1(r)\times \Q_{\tilde x}^1(r)$.
\end{proposition}
 
\begin{proof} We will use, there, the interaction inequality \eqref{eq:intmorV2}  with $N^{HF}_{(2,a^{\star})}(t)= 0$ because of $\beta=0$. Let us start by handling \eqref{eq:nonlinRC}. Namely, by means of
\begin{equation}\label{eq.deltaMor}
\nabla_xa^{\star}(x,y)=\frac{x-y}{|x-y|},
\end{equation}
and inspired by \cite{MXZ}, we can write,
\begin{align}  \label{eq:nonlinRC2}
R^{C}_{(p,|x-y|)}(t)=
\\
\sum_{\substack{j,k,\ell=1}}^N\lambda_{jk}^*\int_{\R^d}\int_{\R^d}\int_{\R^d}
 \frac{(x-y)\cdot(x-z)|\psi_j(x)|^{p}| \psi_k(z)|^{p}}{|x-y||x-z|^{d-\gamma_1+2}} m_{\psi_\ell}(y) \,dxdydz,
 \nonumber\\
=\frac 12\sum_{\substack{j,k=1}}^N\lambda_{jk}^*\int_{\R^d}\int_{\R^d}
 \frac{1}{|x-z|^{d-\gamma_1+2}}|\psi_j(x)|^{p}| \psi_k(z)|^{p}K(x,z)\,dxdz,
 \nonumber
\end{align}
with $\lambda_{jk}^*=8p\lambda_{jk}(d-\gamma_1)/d$ and where
\begin{equation}\label{eq.kernel}
K(x,z)=(x-z)\cdot \sum_{\substack{\ell=1}}^N\int_{\R^d}m_{\psi_\ell}(y)\round{\frac{x-y}{|x-y|}-\frac{z-y}{|z-y|}}\, dy.
\end{equation}
Then, the elementary inequality 
\begin{align*}
(x-z) \cdot \round{\frac{x-y}{|x-y|}-\frac{z-y}{|z-y|}}
\\
=\round{|x-y||z-y|-(x-y)\cdot(z-y)}\round{\frac{|x-y|+|z-y|}{|x-y||z-y|}}\geq 0,
\end{align*}
bears to
\begin{align}\label{eq.liminf}
\inf_{(x,y)\in\R^d\times\R^d }K(x,z)\geq 0.
\end{align}
By combining now the previous \eqref{eq.liminf} with \eqref{eq:nonlinRC2} we obtain that $R^{C}_{(p,a^{\star})}(t)\geq 0$ for any $t\in\R$. Then we achieved, at this stage, the following pointwise (in time) estimate
$$
 N^{C}_{(p,a^{\star})}(t)\leq \ddot {\mathcal I}(t),
$$
which, after an integration w.r.t. time variable over the interval $[t_1, t_2]\subseteq \R$ with $t_1, t_2\in \R$, becomes 
 \begin{align}
\int_{t_1}^{t_2}N^{C}_{(p,|x-y|)}(t)\,dt 
\lesssim \sup_{t\in[t_1, t_2]} |\dot{\mathcal{I}}(t)|.
  \label{eq:stimaMorG1}
\end{align}
We have also the following
\begin{align}\label{eq:Mom}
\sup_{t\in[t_1, t_2]} |\dot{\mathcal{I}}(t)|\leq 2 \sup_{t\in[t_1, t_2]}\sum_{j,\ell=1}^N \left|\int_{\R^d}\int_{\R^d}\je_{\psi_j}(t,x)\cdot\nabla_xa^{\star}(x,y)\m_{\psi_\ell}(t,y)\,dxdy\right|
\nonumber\\
\lesssim\sup_{t\in[t_1, t_2]}\sum_{j=1}^N\|\psi_j(t)\|^4_{H^1_x}
\lesssim\sum_{j=1}^N\|\psi_{j,0}\|^4_{H^1_x}<\infty,
\end{align}
for the reason that the $H^1_x$-norm of the solution is bounded according to the conservation laws
\eqref{eq:macons},and \eqref{eq:enecons}. From the estimates \eqref{eq:Mom}, \eqref{eq:stimaMorG1} and allowing $t_1\to-\infty,  t_2\to+\infty$ we finally get \eqref{eq:stimaMain} which displays, after recalling that
\begin{equation*}
\Delta_x|x-y|=
\begin{cases}
\frac{d-1}{|x-y|}  \ \ \  \  \  \ \ \  \text{if} \ \ \ d\geq2,\\
\\
2\pi\delta_{x=y} \ \  \,  \, \,  \, \, \, \text{if} \ \ \ d=1,
\end{cases}
\end{equation*}
as
\begin{align}  \label{eq:nonlinC23}
\sum_{\substack{j,k,\ell=1}}^N\widetilde\lambda_{jk}\int_{\R}\int_{\R^{3d}}
\frac{d-1}{|x-y||x-z|^{d-\gamma_1}}|\psi_j(x)|^{p}|\psi_\ell(y)|^2| \psi_k(z)|^{p}\,dxdydz,
 \nonumber\\
\lesssim \sum_{j=1}^N\|\psi_{j,0}\|^4_{H^1_x},
\end{align}
with $\R^{3d}=\R^d\times\R^d\times\R^d$, for $d\geq 2$ and
\begin{align}  \label{eq:nonlinC21}
\sum_{\substack{j,k,\ell=1}}^N\widetilde\lambda_{jk}\int_{\R}\int_{\R^{2}}
\frac{1}{|x-z|^{1-\gamma_1}}|\psi_j(x)|^{p}|\psi_\ell(x)|^2| \psi_k(z)|^{p}\,dxdz,
\nonumber\\
\lesssim \sum_{j=1}^N\|\psi_{j,0}\|^4_{H^1_x},
 \end{align}
for $d=1$. We are in position to go over the proof of \eqref{eq:stima0a} and \eqref{eq:stima0b}. We notice that, for any $\tilde x\in \R^d$,

\begin{equation}\label{eq.lbound}
\inf_{x,y,z\in\Q_{\tilde x}^d(r)}\round{\frac{1}{|x-y|}, \frac{1}{|z-y|}}=\inf_{x,y,z\in\Q_0^d(r)}\round{\frac{1}{|x-y|}, \frac{1}{|z-y|}}>0,
\end{equation}
as an outcome, we can bound the l.h.s. of \eqref{eq:nonlinC23} as
\begin{align}\label{eq:nonlinC23a}
\sum_{j,k,\ell=1}^N\widetilde\lambda_{jk}\int_{\R^d}\int_{\R^{3d}}\frac{d-1}{|x-y||x-z|^{d-\gamma_1}}|\psi_j(x)|^{p}|\psi_\ell(y)|^2| \psi_k(z)|^{p}\,dxdydzdt
\nonumber\\
\gtrsim  \sum_{j, k,\ell=1}^N\widetilde\lambda_{jk}\int_{\R}\sup_{\tilde x\in\R^d}\int_{(\Q_{\tilde x}^d(r))^3}|\psi_j(t,x)|^{p}| |\psi_\ell(t,y)|^{2}|\psi_k(t,z)|^{p}
\,dxdydzdt.
\end{align}
Then the previous \eqref{eq:nonlinC23} and \eqref{eq:nonlinC23a}  guarantee that the estimate \eqref{eq:stima0a} holds.
In a similar way we can manage the l.h.s of \eqref{eq:nonlinC21}. To be specific we have, by utilizing again \eqref{eq.lbound}, that
\begin{align}  \label{eq:nonlinC21a}
\sum_{j,k,\ell=1}^N\widetilde\lambda_{jk}\int_{\R}\int_{\R^{2}}\frac{d-1}{|x-y||x-z|^{1-\gamma_1}}|\psi_j(x)|^{p}|\psi_\ell(x)|^2| \psi_k(z)|^{p}\,dxdzdt
\nonumber\\
\gtrsim  \sum_{j, k, \ell=1}^N\widetilde\lambda_{jk}\int_{\R}\sup_{\tilde x\in\R}\int_{(\Q_{\tilde x}^1(r))^2}|\psi_j(t,x)|^{p}| |\psi_\ell(t,x)|^{2}|\psi_k(t,z)|^{p}\,dxdzdt,
\end{align}
The above  \eqref{eq:nonlinC21} and \eqref{eq:nonlinC21a} give the  way to  \eqref{eq:stima0b}. The proof of the proposition is finally completed.
\end{proof}

In addition we get also the following result for the pure HF:

\begin{proposition}\label{dimHF3}
Assume $d\geq3$, $p=2$ and let $(\psi_j)_{i=1}^N\in\mathcal C(\R,H^1(\R^n)^N)$  be as in Proposition \ref{ConsLaw}. 
  Then we have, if one chooses $a^{\star}(x,y)=|x-y|$,
  \begin{align}
-\sum_{j,\ell=1}^N \int_{\R}\int_{\R^d}\int_{\R^d}\Delta^2_xa^{\star}(x,y)|\psi_j(t,x)|^{2}|\psi_\ell(t,y)|^{2}\,dxdydt 
 \leq C\sum_{i=1}^N\|\psi_{i,0}\|^4_{H^1_x}.
  \label{eq:stimaMain2}
\end{align}
Let be $\Q_{\tilde x}^d(r)=\tilde x+[-r,r]^d$, with $r>0$ and $\tilde x\in\R^d$, one gets the following
estimates:
  \begin{itemize}
 \item for $d=3$
 \begin{align} \label{eq:stima1}
  \sum_{j=1}^N \int_{\R}\sup_{\tilde x\in \R^3}\int_{\Q_{\tilde x}^3(r)}|\psi_j(t,x)|^{4}\,dxdt \leq C\sum_{j=1}^N\|\psi_{j,0}\|^4_{H^1_x};
  \end{align}
   \item for $d\geq4$
    \begin{align} \label{eq:stima2}
  \sum_{j,\ell=1}^N   \int_{\R}\sup_{\tilde x\in \R^d}\int_{(\Q_{\tilde x}^d(r))^2}|\psi_j(t,x)|^{2}|\psi_\ell(t,y)|^2\,dxdydt
 \leq C\sum_{j=1}^N\|\psi_{j,0}\|^4_{H^1_x},
    \end{align}
     with $(\Q_{\tilde x}^d(r))^2=\Q_{\tilde x}^d(r)\times \Q_{\tilde x}^d(r)$.
\end{itemize}
 \end{proposition}

\begin{proof}
We notice from the steps above that in the case $p=2$, the term $N^{C}_{(2,|x-y|)}(t)$ will vanish and $R^{C}_{(2,|x-y|)}(t)\geq 0$. We move now on the term \eqref{eq:nonlinHF} having the following
\begin{align}  \label{eq:nonlinHF2}
N^{HF}_{(2,|x-y|)}(t)=
\\
2\beta(d-\gamma_2)\int_{\R^d}\int_{\R^d}
 \frac{1}{|x-z|^{d-\gamma_2+2}}\round{\eta(x) \eta(z)-|\eta(x,z)|^2 }K(x,z)\,dxdy,
 \nonumber
\end{align}
for $K(x,z)$ as in \eqref{eq.kernel} and where we indicated by
\begin{align}\label{eq.HFnonli}
\eta(x,z)=\sum_{j=1}^N \psi_j(x)\bar \psi_j(z), \, \, \,  \, \, \,\eta(x)=\eta(x,x).
\end{align}
In addition we infer, by an use of the Cauchy-Schwartz inequality, the bound $|\eta(x,z)|^2\leq \eta(x)\eta(z)$, for any $x,y\in\R^d$. This observation,  jointly again with \eqref{eq.liminf}, implies $N^{HF}_{(2,|x-y|}(t)\geq 0$ for any $t\in\R$. Hence, by applying the high regularity interaction inequality \eqref{eq:intmor2bb} and then arguing as in \eqref{eq:stimaMorG1} and \eqref{eq:Mom}, one can
easily attain the \eqref{eq:stimaMain2}, which reads, by recalling that
\begin{equation*}
\Delta^2_x|x-y|=-\frac{(d-1)(d-3)}{|x-y|^3},
\qquad
\Delta^2_x|x-y|=-4\pi\delta_{x=y}\leq0,
\end{equation*}
as
 \begin{align*}
\sum_{j,\ell=1}^N \int_{\R}\int_{\R^d}\int_{\R^d}\frac{|\psi_j(t,x)|^{2}|\psi_\ell(t,y)|^{2}}{|x-y|^3}\,dxdt 
 \lesssim \sum_{j=1}^N\|\psi_{\j,0}\|^4_{H^1_x},
  \end{align*}
for $d\geq 4$ and
 \begin{align*}
\sum_{j,\ell=1}^N \int_{\R}\int_{\R^3}|\psi_j(t,x)|^{2}|\psi_\ell(t,x)|^{2}\,dxdt 
 \lesssim \sum_{j=1}^N\|\psi_{j,0}\|^4_{H^1_x},
  \end{align*}
for $d=3$. The proofs of \eqref{eq:stima1} and \eqref{eq:stima2} are exactly the same as in Proposition \ref{dimg3}, 
considering also the bound \eqref{eq.lbound}.
\end{proof}

By the low and high regularity Morawetz interaction inequalities \eqref{eq:intmor2aa}, \eqref{eq:intmor2bb}, the  Propositions \ref{dimg3} and \ref{dimHF3} and taking into account that $N^{C}_{(p,a^{\star})}(t)\geq 0$, one arrives at the following corollary, where some new linear correlation-type estimates associated to \eqref{eq:nls} are achieved. We have thus:

\begin{corollary}\label{corMain}
Let $(\psi_j)_{j=1}^N\in\mathcal C(\R,H^1(\R^n)^N)$  be as in Proposition \ref{ConsLaw}. Then one has, assuming $d\geq1$ and $p>2$ such that \eqref{eq:base} holds,
    \begin{align*}
  \sum^N_{j=1}\|(-\Delta)^{\frac{1-d}{4}}\nabla |\psi_j(t, x)|^2\|^2_{L^2((t_1, t_2);L_x^2)}\lesssim  \sup_{t\in[t_1, t_2]} |\mathcal{ \dot I}(t)|.
  \end{align*}
In particular the following estimates are valid with $p\geq 2$, $\beta\geq 0$ and $\lambda_{jk}\geq 0$:
\begin{itemize}
\item for $d=3$, 
\begin{align*}
  \sum^N_{j=1}\|\psi_j(t, x)\|^4_{L^4((t_1, t_2);L_x^4)}\lesssim  \sup_{t\in[t_1, t_2]} |\mathcal{ \dot I}(t)|;
  \end{align*}
  \item  for $d\geq4$,
  \begin{align*}
  \sum^N_{j=1}\|(-\Delta)^{\frac{3-d}{4}}|\psi_j(t, x)|^2\|^2_{L^2((t_1, t_2);L_x^2)}\lesssim  \sup_{t\in[t_1, t_2]} |\mathcal{ \dot I}(t)|.
  \end{align*}
   \end{itemize}
\end{corollary}

\section{The decay of solutions to \eqref{eq:nls}}\label{MainThm1}

 Our main purpose in this section is to exhibit some decaying properties of the solution to  \eqref{eq:nls} which is a essential property for the study of the scattering phenomena. With the aim of doing that, we present thus the proof of the of Theorem \ref{decay} and of the associated property \eqref{eq:decayHF} in Corollary \ref{thm:mainNLCH2b}.

\begin{proof}[\bf{Proof of Theorem \eqref{decay}.}]
Let us set $u(t,x)=(\psi_j(t,x))_{j=1}^N,$ utilizing both notations where it is needed. We split the proof in three different parts: \\
{\bf Case $p>2, d\geq2.$} 
It is sufficient to prove the property \eqref{eq:decay1} for a suitable $2<r<2d/(d-2)$ (for $2<r<+\infty$, if $d=2$), since the thesis for the general case 
can be acquired by the conservation of mass \eqref{eq:macons}, the kinetic energy 
\eqref{eq:enecons} and then by interpolation. Let us select $r=(2d+8)/(d+2)$, we need to prove then 
\begin{equation}\label{eq:potenergy2}
\lim_{t\rightarrow \pm \infty} \|\psi(t)\|_{\Lin^{\frac{2d+8}{d+2}}_x}=0.
\end{equation}
We treat only the case $t\rightarrow \infty$,
the case $t\rightarrow -\infty$ can be dealt analogously. 
Proceeding now by absurd as in \cite{CT} (we  also to \cite{Vis}), we assume that there exists a sequence
$\{t_n\}$ with $t_n \to +\infty$ and a $\delta_0>0$
\begin{equation}\label{eq:sequencetime}
 \inf_n \|\psi(t_n, x)\|_{\Lin^{\frac{2d+8}{d+2}}_x}=\delta_0.
\end{equation}
Next we will make an use of the localized Gagliardo-Nirenberg inequality given in the Appendix \ref{appen} with $r=1$ and $\nu=2$:
\begin{equation}\label{eq:GNloc3}
\|\phi\|_{\Lin^{\frac{2d+8}{d+2}}_x}^{\frac{2d+8}{d+2}}\leq C \left(\sup_{\widetilde x\in \R^d} \|\phi\|_{\Lin^2(\Q^d_{\widetilde x}(1))}\right)^{\frac{2}{d+4}} 
\|\phi\|^{\frac {d+2}{d+4}}_{\Hin^1_x},
\end{equation}
where $\Q^d_{\widetilde x}(1)$ is the unit cube in $\R^d$ centered in $\widetilde x$.
By combining \eqref{eq:sequencetime}, \eqref{eq:GNloc3}, where we selected $\phi=u(t_n, x)$, with the bound $\|u(t_n, x)\|_{\Hin^1_x}<+\infty$, we notice that there exists $x_n \in \R^d$ and a $\varepsilon_0>0$ such that
\begin{equation}\label{eq:seqspac}
\inf_{n} \|u(t_n, x)\|_{\Lin^2(\Q^d_{x_n}(1))}=\varepsilon_0.
\end{equation}
We can assert now that there exists $t^*>0$ such that
\begin{equation}\label{eq:claMain1}
 \|u(t, x)\|_{\Lin^2(\Q^d_{x_n}(2))}\geq \varepsilon_0/2,
\end{equation}
for all  $t\in (t_n, t_n+ t^*)$ and where $\Q^d_{x_n}(2)$ denotes the cube in $\R^d$ with sidelenght $2$ centered at $x_n$.
Then \eqref{eq:claMain1} can be showed as follows.
Fix a cut-off function $\varphi(x)\in C^\infty_0(\R^d)$, so as
$\varphi(x)=1$ for $x\in \Q^d_{0}(1)$ and $\varphi(x)=0$ for $x\notin \Q^d_{0}(2)$.
Then by applying \eqref{eq:morV} where we choose
$a(x)=\varphi (x-x_n)$ we get
$$\left|\frac d{dt} \int_{\R^d} \varphi(x -x_n) |u(t, x)|^2 dx\right|\lesssim  \sup_t \|u(t,x)\|_{\Hin^1_x}^2.$$
Consequently, by \eqref{eq:enecons} and the fundamental theorem of calculus 
we deduce

\begin{equation}
\left|\int_{\R^d} \varphi(x -x_n) |u(\sigma, x)|^2 dx - \int_{\R^d} \varphi(x -x_n) |u(t, x)|^2 dx\right|\leq \widetilde C |t-\sigma|,
\end{equation}
for a $\widetilde C>0$ which does not depend on $n$. Hence if we choose $t=t_n$ we get
the elementary inequality 

\begin{equation}
\int_{\R^d} \varphi(x -x_n) |u(\sigma, x)|^2 dx\geq  \int_{\R^d} \varphi(x -x_n) |u(t_n, x)|^2 dx - \widetilde C|t_n-\sigma|,
\end{equation}
which implies, having in mind the support property of the function $\varphi$, 

\begin{equation}
\int_{\Q^d_{x_n}(2)} |u(\sigma, x)|^2 dx\geq  \int_{\Q^d_{x_n}(1)} |u(t_n, x)|^2 dx - \widetilde C|t_n-\sigma|.
\end{equation}

Hence \eqref{eq:claMain1} follows by an application of \eqref{eq:seqspac}, provided that we pick up $ t^*>0$
such that
$3\varepsilon_0^2 - 4 \widetilde C t^*>0$.
The inequality \eqref{eq:claMain1} is in contradiction with the Morawetz estimates \eqref{eq:stima0a}.
In fact, the lower bound \eqref{eq:claMain1} means that
\begin{equation}\label{eq:claim1}
\inf_{n}\round{\inf_{t\in (t_n, t_n+t^*)} \sum_{j=1}^N \|\psi_j(t)\|^{2}_{L^{2}_x(\Q^d_{x_n}(2))}} \gtrsim \varepsilon^{2}_0>0,
\end{equation}
with $t^*$ as above and the time intervals $(t_n, t_n+t^*)$ chosen to be disjoint. By H\"older inequality we attain also
\begin{equation}\label{eq:lowp1}
\inf_{n}\round{\inf_{t\in (t_n, t_n+t^*)} \sum_{j=1}^N \|\psi_j(t)\|^{p}_{L^{p}_x(\Q^d_{x_n}(2))}} \gtrsim \varepsilon^{2}_0>0.
\end{equation}
Thus we can formulate the following 
\begin{align} \label{eq:stima3}
\sum_{j, k, \ell=1}^N\widetilde\lambda_{jk}\int_{\R}\sup_{\tilde x\in\R^d}\int_{(\Q_{\tilde x}^d(2))^3}|\psi_j(t,x)|^{p}|\psi_\ell(t,y)|^{2}|\psi_k(t,z)|^{p}
 \,dxdydzdt 
   \nonumber \\
   \gtrsim \sum_{j, k, \ell=1}^N\widetilde\lambda_{jk}\int_{\R} \int_{(\Q_{x_n}^d(2))^3}|\psi_j(t,x)|^{p}|\psi_\ell(t,y)|^{2}|\psi_k(t,z)|^{p}
 \,dxdydzdt 
 	\nonumber\\
 \gtrsim   \sum_{j, k=1}^N\widetilde\lambda_{jk}  \sum_{n} \int_{t_n}^{t_n+t^*} \varepsilon^{6}_0\,dt\gtrsim \sum_{n} t^*\varepsilon^{6}_0\,dt=\infty,&
\end{align}
where in the last inequality  we employed \eqref{eq:claMain1} in combination with \eqref{eq:claim1} and \eqref{eq:lowp1}. This brings us to contradiction with  \eqref{eq:stima0a}. \\
{\bf Case $p>2, d=1$}.  It can be handled  in a similar manner, now by seeking for a $2<r<\infty$. By  an application of the H\"older inequality, one figures out the bound
\begin{equation} 
\begin{split}
  \sum_{j, k, \ell=1}^N\widetilde\lambda_{jk}\int_{\R}\sup_{\tilde x\in\R^1}\int_{(\Q_{\tilde x}^1(2))^2}|\psi_j(t,x)|^{p}|\psi_\ell(t,x)|^{2}|\psi_k(t,z)|^{p}
 \,dxdzdt   \\
  \gtrsim  \sum_{j,k=1}^N\widetilde\lambda_{jk} \sum_{n} \int_{t_n}^{t_n+t^*}\int_{\Q_{ x_n}^1(2)}\int_{\Q_{ x_n}^1(2)}|\psi_j(t,x)|^{p+2}|\psi_k(t,z)|^{p}
  \,dxdzdt
  \nonumber\\
   \gtrsim   \sum_{j,k=1}^N\widetilde\lambda_{jk}  \sum_{n} \int_{t_n}^{t_n+t^*} \varepsilon^{4}_0\,dt\gtrsim \sum_{n} t^*\varepsilon^{4}_0\,dt=\infty.
\end{split}
\end{equation}
Therefore, we can proceed as above, getting a contradiction with \eqref{eq:stima0b} instead. Lastly, the conservation law \eqref{eq:energy}, \eqref{eq:decay1} and the Gagliardo-Nirenberg inequality
$$
\|\psi_j(t)\|^4_{L^{\infty}_x}\lesssim \|\psi_j(t)\|^3_{L^{6}_x} \|\partial_x\psi_j(t)\|_{L^{2}_x}.
$$
ensure  
$$
\lim_{t\rightarrow  +\infty} \|\psi_j(t)\|_{L^\infty_x}=0,
$$
for any $j=1,\dots, N$. \\
{\bf Case $p=2, d\geq 3$}. We follow the same lines of the proof above. However, one can not use, at this level, the Proposition \ref{dimg3} because we are picking up $p=2$. Then we are forced to focus on the Proposition \ref{dimHF3}: for $d\geq 4$, we make use of \eqref{eq:claim1} attaining  
\begin{equation} 
\begin{split}
 \sum_{j, \ell=1}^N\int_{\R}\sup_{\tilde x\in\R^d}\int_{\Q_{\tilde x}^d(2)}\int_{\Q_{\tilde x}^d(2)}|\psi_j(t,x)|^{2}|\psi_\ell(t,y)|^{2}
 \,dxdydt   \\
  \gtrsim  \sum_{j,\ell=1}^N\sum_{n} \int_{t_n}^{t_n+t^*}\int_{\Q_{ x_n}^d(2)}\int_{\Q_{ x_n}^d(2)}|\psi_j(t,x)|^{2}|\psi_\ell(t,y)|^{2}
  \,dxdydt
  \nonumber\\
   \gtrsim   \sum_{n} \int_{t_n}^{t_n+t^*} \varepsilon^{4}_0\,dt= \sum_{n} t^*\varepsilon^{4}_0\,dt=\infty,
\end{split}
\end{equation}
which contradicts \eqref{eq:stima2}. In the same manner we can treat the case $d= 3$. In fact by H\"older inequality and \eqref{eq:claim1}, we arrive at
\begin{equation} 
\begin{split}
 \sum_{j=1}^N\int_{\R}\sup_{\tilde x\in\R^3}\int_{\Q_{\tilde x}^3(2)}|\psi_j(t,x)|^{4}
  \,dxdt   
  \gtrsim  \sum_{j=1}^N\sum_{n} \int_{t_n}^{t_n+t^*}\int_{\Q_{ x_n}^3(2)}|\psi_j(t,x)|^{2}
  \,dxdt
  \nonumber\\
   \gtrsim   \sum_{n} \int_{t_n}^{t_n+t^*} \varepsilon^{2}_0\,dt= \sum_{n} t^*\varepsilon^{2}_0\,dt=\infty,
\end{split}
\end{equation}
that is in contradiction with \eqref{eq:stima2}. Then the proof is now complete.
\end{proof}

 \section{Scattering for NLC and NLHF systems}\label{NLSscat}
 We carry out, along this section, the proof of Theorem \ref{thm:mainNLCH} and the corresponding
 scattering property \eqref{eq:scattering0} in Corollary \ref{thm:mainNLCH2b}.
Albeit these results are classic (we suggest \cite{Ca}, \cite{GiVel2} and references therein for additional reading), 
here we disclose them in a more general and self-contained form. We recall from \cite{KT}, also:
\begin{definition}\label{Sadm}
An exponent pair $(q, r)$ is Schr\"odinger-admissible if $2\leq q,r\leq \infty,$  $(q, r, d ) \neq (2,\infty, 2),$ and
\begin{align}\label{StrTV}
\frac 2q +\frac d{r}=\frac d2.
\end{align}
\end{definition} 

\begin{proposition}\label{Stri}
Let be two Schr\"odinger-admissible pairs $(q,r)$ and $(\widetilde q,\widetilde r)$. Then we have for $\kappa=0,1$ and the following estimates:
\begin{align}
\label{eq:200p1}
 \|  \nabla^{\kappa} e^{ -it\Delta_x} g\|_{L^q_t L^r_x} + \left \|\nabla^{\kappa}\int_0^t e^{- i (t-\tau) \Delta_x} G(\tau)
d\tau\right\|_{L^q_t L^r_x}&\\\nonumber
\leq C\round{\|\nabla^{\kappa} g\|_{ L^2_x}+\|\nabla^{\kappa} G\|_{L^{\widetilde q'}_t
L^{\widetilde r'}_x}}.&
 \end{align}
 \end{proposition}

We want  to prove Theorem \ref{thm:mainNLCH}, then we demand to gain the necessary space-time summability for the scattering. This is contained in the following:
 \begin{lemma}\label{StriNLSys}
 Assume $(\psi_j)_{j=1}^N\in\mathcal C(\R,\Hin^1_x)$ as in Theorem \ref{thm:mainNLCH}. Then we have
 \begin{align}
 (\psi_j)_{j=1}^N\in L^q(\R, \Win^{1,r}_x),
 \end{align}
for every Schr\"odinger-admissible pair $(q,r)$.
\end{lemma}
\begin{proof}
We consider the integral operator associated to \eqref{eq:nls}, that is
\begin{align}\label{eq:opintCH}
u(t)=e^ {it  \Delta_x}u_0 +  \int_{0}^{t} e^ {i(t-\tau)\Delta_x} \Gamma(u(\tau),p) d\tau,
\end{align}
where $t>0$ and 
\begin{gather*}
u(t)=\begin{pmatrix}
      \psi_1(t)\\
      \vdots\\
      \psi_N(t)
      \end{pmatrix}, \quad
u_{0}= \begin{pmatrix}
     \psi_{1,0}\\
      \vdots\\
     \psi_{N,0}
     \end{pmatrix}, \\
   \mathcal G(u,p)=
\begin{pmatrix}  G(\psi_1, \psi_k)\\
\vdots
\\ G(\psi_N, \psi_k) \end{pmatrix}.
\end{gather*}
We start by dealing with $p>2$ and choose 
$( q_1', r_1')$ so that 
\begin{equation}\label{eq:pairCH1}
 ( q_1,r_1):= \left(\frac{4p}{dp-d-\gamma_1},\frac{2dp}{d+\gamma_1}\right).
\end{equation}
In this way the Strichartz estimates \eqref{eq:200p1}, the fractional chain rule, the H\"older and Hardy-Littlewood-Sobolev inequalities enhance, for $\kappa=0,1$, to the following (see \cite{MXZ})

\begin{align}\label{eq.1nonl}
\sum_{\kappa=1}^2\norm{\nabla^{\kappa}\psi_j}{L_t^{q_1}L_x^{r_1}}\\
\lesssim\sum_{\kappa=1}^2 \norm{\nabla^{\kappa}\psi_{j,0}}{L_x^{2}} +\sum_{\kappa=1}^2\left\|\sum_{k=1}^N \lambda_{jk}\nabla^{\kappa}\round{ \squad{|x|^{-(d-\gamma_1)}*| \psi_k|^p} | \psi_j|^{p-2}  \psi_j}\right\|_{L^{q_1'}_{t>T}L^{r_1'}_x}
\nonumber\\
\lesssim \sum_{\kappa=1}^2 \norm{\nabla^{\kappa}\psi_{j,0}}{L_x^{2}}+\sum_{\kappa=1}^2\norm{ \sum^N_{k=1}\lambda_{jk}\|\nabla^{\kappa}\psi_j\|_{L^{r_1}_x}\|\psi_k\|_{L_x^{r_1}}^{p} \|\psi_j\|_{L_x^{r_1}}^{p-2}}{L^{q_1'}_{t>T}}
\nonumber\\
+\left\|\sum_{k=1}^N \lambda_{jk}  \squad{|x|^{-(d-\gamma_1)}*\nabla| \psi_k|^p} | \psi_j|^{p-2}  \psi_j\right\|_{L^{q_1'}_{t>T}L^{r_1'}_x}
\nonumber\\
\lesssim \sum_{\kappa=1}^2 \norm{\nabla^{\kappa}\psi_{j,0}}{L_x^{2}}+\sum_{\kappa=1}^2\norm{ \sum^N_{k=1}\lambda_{jk}\|\nabla^{\kappa}\psi_j\|_{L^{r_1}_x}\|\psi_k\|_{L_x^{r_1}}^{p} \|\psi_j\|_{L_x^{r_1}}^{p-2}}{L^{q_1'}_{t>T}}
\nonumber\\
+\left\|\sum_{k=1}^N \lambda_{jk} \|\nabla\psi_k\|_{L_x^{r_1}}\|\psi_k\|_{L_x^{r_1}}^{p-1}\|\psi_j\|_{L_x^{r_1}}^{p-1}\right\|_{L^{q_1'}_{t>T}}.
\nonumber
\end{align} 
 Summing up over $j=1\dots N$, we see that the last term of the previous inequality is not greater than
  \begin{equation}\label{eq.1non2fNL}
 \begin{split}
   \sum^N_{j=1}\sum_{\kappa=1}^2 \norm{\nabla^{\kappa}\psi_{j,0}}{L_x^{2}}+\sum^N_{j,k=1}\sum_{\kappa=1}^2\norm{\lambda_{jk}\|\nabla^{\kappa}\psi_j\|_{L^{r_1}_x}\|\psi_k\|_{L_x^{r_1}}^{p} \|\psi_j\|_{L_x^{r_1}}^{p-2}}{L^{q_1'}_{t>T}}
\\
+\sum^N_{j,k=1}\left\| \lambda_{jk} \|\nabla\psi_k\|_{L_x^{r_1}}\|\psi_k\|_{L_x^{r_1}}^{p-1}\|\psi_j\|_{L_x^{r_1}}^{p-1}\right\|_{L^{q_1'}_{t>T}L^{r_1'}_x}.
\\
 \lesssim \|u_0\|_{\Hin^1_x} +\sum_{\kappa=1}^2\Big\|
    \|\nabla^{\kappa}u\|_{\Lin^{r_1}_x}
   \|u\|_{\Lin_x^{r_1}}^{(2p-2)}\Big\|_{L^{q'_1}_{t>T}}&
\end{split}
\end{equation}
We single out now $\theta_1=(q_1-q_1')/(2pq_1'-2q'_1)\in(0,1),$ because of \eqref{eq:base} and \eqref{eq:baseII}. Furthermore, direct calculations show that 
$$
\frac 1{q_1'}=\frac{2(p-1)\theta_1+1}{q_1},
$$
yielding for the last term  on the r.h.s. of
\eqref{eq.1non2fNL},
 \begin{equation}\label{eq.1non2s}
 \begin{split}
 \|u_0\|_{\Hin^1_x} +\sum_{\kappa=1}^2\Big\|
    \|\nabla^{\kappa}u\|_{\Lin^{r_1}_x}
   \|u\|_{\Lin_x^{r_1}}^{(2p-2)}\Big\|_{L^{q'_1}_{t>T}}
   \\
   \lesssim
   \|u_0\|_{\Hin^1_x}+ \Big\|  \|\nabla^{\kappa}u\|_{\Lin^{r_1}_x}
   \|u\|_{\Lin_x^{r_1}}^{(2p-2)(1-\theta_1)}\|u\|_{\Lin_x^{r_1}}^{(2p-2)\theta_1}\Big\|_{L^{q'_1}_{t>T}}&
   \\
    \lesssim 
   \|u_0\|_{\Hin^1_x}+ \Big\|  \|u\|_{\Win^{1,r_1}_x} 
   \|u\|_{\Lin_x^{r_1}}^{(2p-2)(1-\theta_1)}\|u\|_{\Lin_x^{r_1}}^{(2p-2)\theta_1}\Big\|_{L^{q'_1}_{t>T}}&
  \\
   \lesssim 
   \|u_0\|_{\Hin^1_x}+ \Big\|  
   \|u\|_{\Lin_x^{r_1}}^{(2p-2)(1-\theta_1)}\|u\|_{\Win^{1,r_1}_x}^{(2p-2)\theta_1+1}\Big\|_{L^{q'_1}_{t>T}}&
   \\
   \leq C\round{  \|u_0\|_{\Hin^1_x}+ \|u\|_{L^{\infty}_{t>T}\Lin_x^{r_1}}^{(2p-2)(1-\theta_1)}
  \|u\|_{L^{q_1}_{t>T}\Win^{1,r_1}_x}^{(2p-2)\theta_1+1}}&,
\end{split}
\end{equation}
with the constant $C>0$ independent from $t$ and $T$. An use of  \eqref{eq.1nonl}, \eqref{eq.1non2fNL},  \eqref{eq.1non2s}
leads to
\begin{align*}
\|u\|_{L^{q_1}_{t>T} \Win^{1,r_1}_x}
\lesssim\|u_0\|_{\Hin^1_x}+ \|u\|_{L^{\infty}_{t>T}\Lin_x^{r_1}}^{(2p-2)(1-\theta_1)}\|u\|_{L^{q_1}_{t>T}\Win^{1,r_1}_x}^{(2p-2)\theta_1+1},
\end{align*}
where 
$$
\lim _{T\rightarrow +\infty}\|u\|_{L^{\infty}_{t>T}\Lin_x^{r_1}}=0
$$
by \eqref{eq:decay1} in Theorem \ref{decay}. Then, picking up $T$ sufficiently large we infer that
$$
\|u\|_{L^{q_1}((T,t),\Win^{1,r_1}_x)}<\infty ,
$$
and consequently that $u\in L^{q_1}((T,+\infty), \Win^{1,r_1}_x).$ Likewise, we can earn $u\in L^{q_1}((-\infty, -T), \Win^{1,r_1}_x).$
In conclusion, by a continuity argument and Strichartz estimates \eqref{StrTV}, one has  $u\in L^q(\R, \Win^{1,r}_x)$
for any Schr\"odinger-admissible pair $(q,r)$.\\
Let us manage $p=2$. We pick $( q_2', r_2')$ defined by
\begin{equation}\label{eq:pairCH2}
 ( q_2,r_2):= \left(\frac{8}{d-\gamma_1},\frac{4d}{d+\gamma_1}\right),
\end{equation}
then we get, analogously as above, 
\begin{equation}\label{eq.1nonl2}
\begin{split}
 \sum^N_{j=1}\sum_{\kappa=1}^2\norm{\nabla^{\kappa}\psi_j}{L_t^{q_2}L_x^{r_2}}\lesssim   \sum^N_{j=1}\sum_{\kappa=1}^2 \norm{\nabla^{\kappa}\psi_{j,0}}{L_x^{2}}&
\\
+  \sum^N_{j,k=1}\sum_{\kappa=1}^2\left\| \nabla^{\kappa}\round{\lambda_{jk} \squad{|x|^{-(d-\gamma_1)}*| \psi_k|^2}   \psi_j}\right\|_{L^{q_2'}_{t>T}L^{r_2'}_x}&
\\
-\beta \sum^N_{j,k=1}\sum_{\kappa=1}^2\left\| \nabla^{\kappa}\round{ \squad{|x|^{-(d-\gamma_2)}*| \psi_k|^2}    \psi_j-\squad{|x|^{-(d-\gamma_2)}*\bar \psi_k \psi_j  } \psi_k}\right\|_{L^{q_2'}_{t>T}L^{r_2'}_x}&
\\
\lesssim \|u_0\|_{\Hin^1_x} +\sum_{\kappa=1}^2\Big\|
    \|\nabla^{\kappa}u\|_{\Lin^{r_2}_x}
   \|u\|_{\Lin_x^{r_2}}^{2}\Big\|_{L^{q'_2}_{t>T}}&
\\
   \lesssim 
   \|u_0\|_{\Hin^1_x}+ \Big\|  
   \|u\|_{\Lin_x^{r_2}}^{2(1-\theta_2)}\|u\|_{\Win^{1,r_2}_x}^{2\theta_2+1}\Big\|_{L^{q'_2}_{t>T}}&
   \\
   \lesssim  \|u_0\|_{\Hin^1_x}+ \|u\|_{L^{\infty}_{t>T}\Lin_x^{r_2}}^{2(1-\theta_2)}
  \|u\|_{L^{q_2}_{t>T}\Win^{1,r_2}_x}^{2\theta_2+1},&
\end{split}
\end{equation} 
where $\theta_2=(q_2-q_2')/2q_2'\in(0,1)$ and such that
$$
\frac 1{q_2'}=\frac{2\theta_1+1}{q_2}.
$$
This enables us to rewrite \eqref{eq.1nonl2} as
\begin{align*}
\|u\|_{L^{q_2}_{t>T} \Win^{1,r_2}_x}
\lesssim\|u_0\|_{\Hin^1_x}+ \|u\|_{L^{\infty}_{t>T}\Lin_x^{r_2}}^{2(1-\theta_2)}\|u\|_{L^{q_2}_{t>T}\Win^{1,r_2}_x}^{2\theta_2+1},
\end{align*}
with 
$$
\lim _{T\rightarrow +\infty}\|u\|_{L^{\infty}_{t>T}\Lin_x^{r_2}}=0
$$
again by \eqref{eq:decay1} in Theorem \ref{decay}. Thus one argues as in the previous lines
carrying out again that $u\in L^q(\R, \Win^{1,r}_x)$
for any admissible pair $(q,r)$.
\end{proof}

\begin{proof}[Proof of Theorem \ref{thm:mainNLCH}]
 We exploit the proof of Theorem \ref{thm:mainNLCH} for $p>2$ and $p=2$ in a unified manner. We start from:
   \\
  {\em Asymptotic completeness:} We write $\widetilde  u(t)=e^ {-it \Delta_{x}}u(t)$ getting then from \eqref{eq:opintCH} 
  \begin{align*}
\widetilde  u(t_2)-\widetilde  u(t_1)=  i\int_{t_1}^{t_2} e^ {-is \Delta_{x}}  \mathcal G(u,p)ds.
\end{align*}
An use of the Strichartz estimates \eqref{StrTV} bears to
 \begin{align}\label{eq:stric1}
\norm{\int_{t_1}^{t_2} e^ {-is \Delta_{x}}  \mathcal G(u,p)ds}{\Hin^1_x}
\nonumber\\
\lesssim 
\sum_{j,k=1}^N\lambda_{jk}\left\|\squad{|x|^{-(d-\gamma_1)}*| \psi_k|^p} | \psi_j|^{p-2}  \psi_j\right\|_{L^{q_1'}((t_1,t_2),W^{r_1'}_x)}
\\
+\beta\sum_{j,k=1}^N\left\| \squad{|x|^{-(d-\gamma_2)}*| \psi_k|^2}    \psi_j
-\squad{|x|^{-(d-\gamma_2)}*\bar \psi_k \psi_j  } \psi_k\right\|_{L^{q_2'}((t_1,t_2), W^{r_2'}_x)},
\nonumber
\end{align}
with $(q_1, r_1)$ and $(q_2, r_2)$ admissible pairs as in \eqref{eq:pairCH1} and \eqref{eq:pairCH2}. 
Then it suffices to display that
 \begin{align*}
\lim_{t_1, t_2\rightarrow \infty}\|\widetilde  u(t_2)-\widetilde  u(t_1)\|_{\Hin^1_x}=0,
\end{align*}
which is verified by \eqref{eq:stric1} on condition that  
 \begin{align*}
\lim_{t_1, t_2\rightarrow \infty}\sum_{j,k=1}^N\lambda_{jk}\left\|\squad{|x|^{-(d-\gamma_1)}*| \psi_k|^p} | \psi_j|^{p-2}  \psi_j\right\|_{L^{q_1'}((t_1,t_2), W^{r_1'}_x)}=0
\\
\lim_{t_1, t_2\rightarrow \infty}\sum_{j,k=1}^N\left\| \squad{\round{\lambda_{jk}|x|^{-(d-\gamma_1)}+\beta|x|^{-(d-\gamma_2)}}*| \psi_k|^2}    \psi_j\right\|_{L^{q_2'}((t_1,t_2), W^{r_2'}_x)}=0
\nonumber\\
\beta\lim_{t_1, t_2\rightarrow \infty}\sum_{\substack{j,k=1}}^N\left\|  \squad{|x|^{-(d-\gamma_2)}*\bar \psi_k \psi_j } \psi_k\right\|_{L^{q_2'}((t_1,t_2), W^{r_2'}_x)}=0,
\nonumber
\end{align*}
which can be easily performed following the same lines of the proof of Lemma \ref{StriNLSys}.
One can see, as a final step, that there are $(\psi_{1,0}^{\pm},\dotsc, \psi_{N,0}^\pm)\in H^1(\R^d)^N$ and a map
 $(\psi_1(t),\dotsc, \psi_N(t))\rightarrow (\psi_{1,0}^{\pm},\dotsc, \psi_{N,0}^\pm)$ in $H^1(\R^d)^N$ when $t\rightarrow\pm \infty.$
Notice that, by Proposition \ref{ConsLaw}, we establish also  the following conservation laws
\begin{equation*}
\begin{split}
M(\psi_{1,0}^{\pm},\dotsc,\psi_{N,0}^\pm)=\|(\psi_{1,0},\dotsc, \psi_{N,0})\|^2_{\Lin^2_x}, \\ 
 \sum_{j=1}^N  \int_{\R^d} \Big(\abs{\Delta{\psi_{j,0}^{\pm}}}^2+\kappa|\nabla \psi_{j,0}^{\pm}|\Big) dx=E(\psi_{1,0},\dotsc, \psi_{N,0}).
\end{split}
\end{equation*}

  {\em Existence of wave operators:} The construction of the wave operators comes from standard arguments, we refer to \cite{Ca} for more details about the matter. Then we skip the proof.
 \end{proof}

\begin{remark}
Once \eqref{eq:decay1} is achieved in the range $2<r< 2d/(d-2)$, we were able to  set up the scattering operator in $H^1(\R^d)^N$, as we did in the previous section. Now, similarly  \cite{Vis}, we arrive by
Sobolev embedding at
 \begin{align}\label{eq:stric2V}
\norm{\psi_j(t)}{L_x^\frac{2d}{d-2}}
\lesssim
\left\|\psi_j(t)-e^{it\Delta}\psi_{j,0}^\pm\right\|_{H_x^1}+\norm{e^{it\Delta}\psi_{j,0}^\pm}{L_x^\frac{2d}{d-2}}.
\end{align}
Now the above estimate \eqref{eq:stric2V} combined with the classical dispersive estimate for the free propagator
\begin{align*}
\norm{e^{it\Delta}\psi_{j,0}^\pm}{L_x^\frac{2d}{d-2}}
\lesssim
\frac 1 t \norm{\psi_{j,0}^\pm}{L_x^\frac{2d}{d+2}},
\end{align*}
again the Sobolev-embedding and \eqref{eq:scattering}, allow also to 
 \begin{align*}\label{eq:stric2V}
\lim_{t\rightarrow \infty}\norm{\psi_j(t)}{L_x^\frac{2d}{d-2}}=0.
\end{align*}
The proof of Theorem \ref{decay} is now completed.
\end{remark}

\appendix
\section{A Gagliardo-Nirenberg inequality}\label{appen}
The principal target of this section is to exhibit \eqref{eq:GNloc3} that is a localized version of the Gagliardo-Nirenberg inequality which appears in the proof of
Proposition \ref{decay}. Although it is known so far
in the literature in different forms (let us cite here \cite{CT}, \cite{Vis}, \cite{L1} or \cite{TeTzVi} in the context of product space $\R^d\times M,$ with $M$ a compact manifold), we show here a more general new one. We have:
\begin{proposition}\label{GNloc}
Let be $d\geq 1$, $\mu \in \N$ and $\nu \in \N\cup \{0\}$, then for all vector-valued functions $\phi=(\phi_\ell)_{\ell=1}^\mu \in H^1(\R^d)^\mu$ one gets the following 
\begin{equation}\label{eq:GNlocN}
\|\phi\|_{L^{\frac{2d+2\nu+4}{d+\nu}}(\R^d)^\mu}^{\frac{2d+2\nu+4}{d+\nu}}\leq C \left(\sup_{x\in \R^d} \|\phi\|_{ L^2(Q^d_{\widetilde x}(r))^\mu}\right)^{\frac{4}{d+\nu}} 
\|\phi\|^2_{ H^1(\R^d)^\mu},
\end{equation}
with $\mathcal Q_{\widetilde x}^d(r)=\widetilde x+[-r,r]^d\,$ being a $r$ dilation of the unit cube centered at $\widetilde x$.
\end{proposition}
\begin{proof}
Fix $r>0$ and consider $\widetilde x_s\in\R^d$  connected to a covering of $\R^d$ given by a family of cubes $\{\mathcal Q_{\widetilde x_s}^d(r)\}_{s\in\N}$ such that $meas_d\round{ \Q_{\widetilde x_{s_1}}^d(r)\cup \Q_{\widetilde x_{s_2}}^d(r)}=0$ for $s_1\neq s_2$, where $meas_d$ is the Lebesgue measure in $\R^d$. Without loss of generality, we can take $\phi=(\phi_\ell)_{\ell=1}^\mu$, such that $\supp (\phi_\ell) \subseteq \Q_{\widetilde x_s}^d(r)$ for any $ \ell=1,\dots,\mu$, then by the classical Gagliardo-Nirenberg inequality we 
attain

\begin{align}\label{eq:GNlocN2}
\sum_{\ell=1}^\mu\int_{\Q_{\widetilde x_s}^d(r)}|\phi_\ell|^{\frac{2d+2\nu+4}{d+\nu}}\lesssim \sum_{\ell=1}^\mu\left(\int_{\Q_{\widetilde x_s}^d(r)}|\phi_\ell|^{\rho}\right)^{
\frac 4{\rho(d+\nu)}} 
\left(\int_{\Q_{\widetilde x_s}^d(r)}|\nabla\phi_\ell|^{\rho}\right)^{\frac{2}{\rho}}&
\end{align}
and
$$
\rho=\frac{2d(d+\nu+2)}{(d+2)(d+\nu)}>2.
$$
An application of the H\"older inequality, gives that the r.h.s. of \eqref{eq:GNlocN2} is bounded as

\begin{align}\label{eq:GNlocN2a}
 \sum_{\ell=1}^\mu\left(\int_{\Q_{\widetilde x_s}^d(\rho)}|\phi_\ell|^{\rho}\right)^{
\frac 4{\rho(d+\nu)}} 
\left(\int_{\Q_{\widetilde x_s}^d(r)}|\nabla\phi_\ell|^{\rho}\right)^{\frac{2}{\rho}}&
\\
\leq C\sum_{\ell=1}^\mu\left(\int_{\Q_{\widetilde x_s}^d(r)}|\phi_\ell|^{2}\right)^{
\frac 2{d+\nu}} 
\left(\int_{\Q_{\widetilde x_s}^d(r)}|\nabla\phi_\ell|^{2}\right)
\nonumber\\
\leq C \left(\sum_{\ell=1}^\mu\|\phi_\ell\|_{ L^2(\Q_{\widetilde x_s}^d(r))}\right)^{
\frac 4{d+\nu}} 
\sum_{\ell=1}^\mu\|\phi_\ell\|^2_{H^1(\Q_{\widetilde x_s}^d(r))}.&
\nonumber
\end{align}
with $C>0$ a constant depending on $meas_d\round{\Q_{\widetilde x_s}^d(r)}$. From \eqref{eq:GNlocN2} and \eqref{eq:GNlocN2a}
one can get 
\begin{equation}\label{eq:GNlocN3}
\|\phi\|_{L^{\frac{2d+2\nu+4}{d+\nu}}(\Q_{\widetilde x_s}^d(r))^\mu}^{\frac{2d+2\nu+4}{d+\nu}}\leq C \left(\|\phi\|_{L^2(\Q_{\widetilde x_s}^d(r))^\mu}\right)^{\frac 4{d+\nu}} 
\|\phi\|^2_{H^1(\Q_{\widetilde x_s}^d(r))^\mu}.
\end{equation}
Hence summing over $s$ we obtain
 \begin{align}\label{eq:GNlocN4}
\|\phi\|_{L^{\frac{2d+2\nu+4}{d+\nu}}(\R^d)^\mu}^{\frac{2d+2\nu+4}{d+\nu}}\leq C \left(\sup_{s\in\N}\|\phi\|_{L^2(\Q_{\widetilde x_s}^d(r))^\mu}\right)^{\frac{4}{d}} 
\sum_{s\in\N}\|\phi\|^2_{H^1(\Q_{\widetilde x_s}^d(r))^\mu}&
\nonumber\\
\leq C \left(\sup_{x\in\R^d}\|\phi\|_{L^2(\Q_{\widetilde x}^d(r))^\mu}\right)^{\frac 4{d+\nu}} 
\|\phi\|^2_{H^1(\R^d)^\mu},&
\end{align}
which is the estimate \eqref{eq:GNlocN}, with the constants involved  independent from $s$ because the estimate above is translation invariant.

\end{proof}

\end{document}